\documentclass[11pt]{article}

\usepackage{amsfonts,amssymb}

\usepackage{amsthm}

\usepackage{amsmath}

\usepackage{amscd}
\usepackage{tabularx}
\usepackage{arydshln}
\usepackage{verbatim}
\usepackage{color}
\usepackage{a4wide}

\usepackage{tikz,tikz-cd} 
 
\usepackage[all]{xy}

\theoremstyle{plain}

\usepackage{tikz-cd}
\tikzset{commutative diagrams/.cd}
\usetikzlibrary{intersections}
\usepackage{ mathrsfs }

\newtheorem{theorem}{Theorem}[section]
\newtheorem{lemma}[theorem]{Lemma}
\newtheorem{proposition}[theorem]{Proposition}

\newtheorem{Counter-example}[theorem]{Counter-example}
\newtheorem{remark}[theorem]{Remark}

\theoremstyle{definition}

\newtheorem{definition}[theorem]{Definition}

\theoremstyle{remark}

\usepackage{amssymb}

\def\R{\mathbb R}

\def\S{\mathbb S}
\def\L{\mathbb L}

\def\E{\mathbb E}

\begin{document}

\title{Cartan geometries with model the future lightlike cone of Lorentz-Minkowski spacetime}

\author{ 
Rodrigo Morón\footnote{Both authors were partially supported by the Spanish MICINN and ERDF project PID2024-156031NB-I00.
}\\ 
 \texttt{Departamento de Matemáticas, Universidad de León} \\ 
 \texttt{$rmors@unileon.es$} 
 \and 
  Francisco J. Palomo$^*$\\ 
 \texttt{Departamento de Matemática Aplicada, Universidad de Málaga} \\ 
 \texttt{$fpalomo@uma.es$} 
\date{\empty}
} 


\maketitle

\begin{abstract}
\noindent This paper develops the theory of Cartan geometries modeled on the future lightlike cone of Lorentz–Minkowski spacetime, which we refer to as lightlike Cartan geometries. We show that such geometries naturally induce on the base manifold a lightlike metric, a globally defined radical vector field, and two additional compatible structures. Within this framework, we construct the standard tractor bundle associated with a lightlike Cartan geometry, showing that it extends the tangent bundle of the base manifold and carries a canonical metric linear connection. This construction provides an alternative characterization of lightlike Cartan geometries purely in terms of vector bundle data. Using this alternative description, we analyze the additional geometric information encoded in the Cartan connection beyond the metric and radical data, and we show how it can be extracted via natural decompositions of the standard tractor bundle. Our results underscore the intrinsic significance of Cartan connection methods in the study of lightlike geometry and open new avenues for the analysis of geometric structures that are neither parabolic nor reductive.
\end{abstract}


\section{Introduction}

Cartan geometries provide a unified description of a wide variety of geometric structures. The basic idea is to associate a geometry to each homogeneous space, using it as a model. In this sense, Cartan geometries can be seen as the curved analogues of homogeneous spaces, see Section \ref{31102024B}. Very often, the description of the underlying structures on the base manifold of a principal fiber bundle equipped with Cartan connections is not easy to interpret geometrically, and determining a canonical Cartan connection for a given geometric structure is often a central and highly non-trivial problem.

This article is devoted to a specific case within this broad general theory, namely the study of the underlying geometric structure for Cartan geometries modeled on the future lightlike cone of Lorentz-Minkowski spacetime. As was shown in \cite{PacoLuz}, these Cartan geometries induce a lightlike metric and a vector field that spans its radical distribution, see also Section \ref{13333082025A}. However, as will be discussed later, these Cartan geometries are required to endow the manifold with additional geometric structures. These structures are described in Section \ref{14082025847564857}. Our interest in these Cartan geometries stems from the fact that they may be useful for studying local invariants in the intrinsic geometry of lightlike hypersurfaces.

\vspace{2mm}

At this point, let us recall that the boundaries of certain spacetimes are often lightlike hypersurfaces, such as Killing horizons and the conformal boundaries of asymptotically flat spacetimes. From a geometric perspective, a lightlike hypersurface \(N\) of a Lorentzian manifold inherits from the ambient metric tensor a degenerate bilinear form \(h\), which is positive but not definite. Its radical,
\[
\mathrm{Rad}(h) := \{ v \in TN : h(v, \cdot) = 0 \},
\]
defines a one-dimensional distribution on \(N\). Due to the existence of the radical distribution \(\mathrm{Rad}(h)\), the classical Koszul formula cannot be used to determine a linear connection on \(N\). Therefore, the so-called miracle of semi-Riemannian geometry does not work in this case. Moreover, the normal bundle of a lightlike hypersurface intersects nontrivially with its tangent bundle. Consequently, there is no natural direct sum decomposition of the tangent bundle of the ambient Lorentzian manifold along the immersion, and the classical theory of non-degenerate submanifolds does not apply in this setting.

\vspace{2mm}

These difficulties mainly arise in the study of the extrinsic geometry of lightlike hypersurfaces. Without attempting to be exhaustive, we may mention the general geometric technique developed by Bejancu and Duggal in \cite{DB96} to deal with such hypersurfaces. Their approach relies on the notion of a screen distribution: they define a non-degenerate distribution on the lightlike hypersurface $N$, starting from a transversal lightlike vector field along the hypersurface, see Section \ref{14082025847564857} for a more detailed comparison.
Degenerate submanifolds of semi-Riemannian manifolds endowed with additional geometric structures are discussed in \cite{Kupeli}. We also recall the quotient-type construction presented in \cite[Chap.~4]{One83}; see Remark~\ref{14082025hgjdkld} for further details. These ideas have led to significant results, for instance in \cite{Galloway} and \cite{Galloway2}.
From the intrinsic point of view, the local equivalence problem has been addressed in the case of three-dimensional lightlike manifolds in \cite{NuRo}. 

\vspace{2mm}

The viewpoint adopted here is intrinsic, describing lightlike manifolds—as defined in Definition $\ref{52}$—as a particular type of Cartan geometry. In this sense, the present work may be regarded as a continuation of \cite{PacoLuz} and, to a large extent, as complementary to it. 
The main goal of this paper is to provide a complete description of the structure underlying a Cartan geometry modeled on the future cone of Lorentz-Minkowski spacetime.
Later on, we shall highlight the differences with \cite{PacoLuz}, as well as the new contributions introduced here. Although this is not the perspective adopted in this article, lightlike manifolds are closely related to Carrollian geometries and other non-Lorentzian descriptions of spacetime. Models for such structures can be found in \cite{Fi}, \cite{Herfray}, and the references therein.

\vspace{2mm}

Our first task is to justify the choice of the model for the Cartan geometries considered in this paper: the future lightlike cone $\mathcal{N}^{m+1}$ in Lorentz-Minkowski spacetime $\L^{m+2}$. The choice of the future lightlike cone is motivated by the rigidity result \cite[Theorem 1.1]{BFZ}, see Section \ref{31102024A}. Among other statements, this result shows that for $m\geq 2$, the isometry group of $\mathcal{N}^{m+1}$, with respect to the lightlike metric inherited from $\L^{m+2}$, is the orthochronous Lorentz group $O^{+}(m+1, 1)$. By contrast, for a lightlike hyperplane $\Pi \subset \L^{m+1}$, the isometry group is infinite-dimensional, and hence $\Pi$ cannot serve as the model of a Cartan geometry for lightlike manifolds \cite[Example 4.13]{PacoLuz}. On the other hand, it is well known that the orthochronous Lorentz group plays a key role in the description of Riemannian  conformal geometry as a Cartan geometry. The close relationship between these two settings may be summarized by the following diagram (see Section \ref{31102024A} for details):
$$
    \begin{tikzpicture}

	\node (A) at (0,0) {$\mathcal{N}^{m+1}=O^{+}(m+1,1)/H$};
	\node (B) at (7,0) {$\S^{m}=O^{+}(m+1,1)/P$};
	\node (C) at (3.5,1.5) { $O^{+}(m+1,1)$};
	
	\draw[->] (A) -- node[midway, below] {$\pi$} (B);
	\draw[<-] (B) --  (C);
	\draw[->] (C) --  (A);
\end{tikzpicture}
$$
Here $H$ is isomorphic to the group of rigid motions of the Euclidean space $\E^{m}$, while $P$ denotes the Poincar\'{e} conformal group. The projection $\pi$ is a principal fiber bundle with structure group $\mathbb{R}_{>0} \simeq P/H$, and every section $s \in \Gamma(\mathcal{N}^{m+1})$ induces, via pullback, a Riemannian metric on the sphere $\mathbb{S}^m$ that is conformal to the standard round metric. This illustrates the strong connection between lightlike manifolds described as Cartan geometries modeled on the lightlike cone and the so-called fiber bundles of scales in Riemannian conformal geometry \cite[Section 1.6]{CS09}. Roughly speaking, the future lightlike cone $\mathcal{N}^{m+1}$ may be regarded as the bundle of rays formed by the metrics in the conformal class of the round metric on $\mathbb{S}^m$.

The future lightlike cone, when formulated as a homogeneous space as above, gives rise to a first-order Cartan geometry. The associated Cartan principal fiber bundle admits a natural interpretation as a bundle of adapted frames. It is worth recalling that the parabolic and reductive cases form the most thoroughly studied classes of Cartan geometries, supported by a well-developed and extensive theoretical framework \cite{CS09}. In contrast, the model considered here is neither parabolic nor reductive. To the best of our knowledge, this is the first systematic attempt to investigate the geometric significance of Cartan geometries of this type.

\vspace{2mm}

Recall that each Cartan geometry \((p: \mathcal{P} \to N, \omega)\) modeled on the future lightlike cone equips the manifold \(N\) with a lightlike metric \(h^{\omega}\), together with a vector field 
\(Z^{\omega} \in \mathfrak{X}(N)\) that spans the radical distribution of \(h^{\omega}\) \cite[Theorem 4.4]{PacoLuz}. The point of view adopted in this paper, similar in spirit to the framework developed for parabolic geometries in the so-called tractor calculus \cite{TC1}, is to describe such Cartan geometries equivalently in terms of the tractor vector bundles \(\mathcal{T} \to N\), obtained from the standard representation of \(O^{+}(m+1,1)\) on \(\mathbb{L}^{m+2}\), see Definition~\ref{23102024A}. A crucial observation is that a standard tractor bundle \(\mathcal{T} \to N\) may be regarded as a natural metric extension of the tangent bundle \(TN\) and is, in addition, canonically endowed with a metric-compatible linear connection \(\nabla^{\mathcal{T}}\), see Remark~\ref{220825A}. Roughly speaking, for each point \(x \in N\), the fiber \(\mathcal{T}_x\) carries a Lorentzian inner product and contains \(T_xN\) as a degenerate hyperplane embedded isometrically. This structural feature plays a central role in the developments and results presented in this paper.

\vspace{2mm}

One general feature of Cartan geometries is that the automorphism group of any Cartan geometry on a connected manifold is always a Lie group \cite[Theorem 1.5.11]{CS09}. As noted in \cite{PacoLuz}, the automorphism group \(\mathrm{Aut}(\mathcal{P}, \omega)\) of a Cartan geometry modeled on the lightlike cone over a connected manifold \(N\) is, in general, a proper subgroup of \(\mathrm{Iso}(N, h^{\omega}, Z^{\omega})\). Indeed, every \(f \in \mathrm{Aut}(\mathcal{P}, \omega)\) preserves both the lightlike metric \(h^{\omega}\) and the vector field \(Z^{\omega}\), whereas the group \(\mathrm{Iso}(N, h^{\omega}, Z^{\omega})\) may be infinite-dimensional. This naturally raises the question of whether further geometric structures can be canonically induced from a Cartan geometry modeled on the lightlike cone. This issue is addressed in Theorem~\ref{19082025hfjrfhgfujfg}. 
\begin{quote} 
Let \((p: \mathcal{P} \to N, \omega)\) be a Cartan geometry modeled on the future lightlike cone. Then the manifold \(N\) is naturally endowed with the following structures: 
\begin{enumerate} 
\item[$(1)$] A lightlike metric \(h^{\omega}\) together with a vector field \(Z^{\omega} \in \mathfrak{X}(N)\) spanning the radical distribution of \(h^{\omega}\). 
\item[$(2)$] A tensor field \(\mathbf{T}^{\omega} \in \Gamma(\Lambda^2 T^*N \otimes TN)\). 
\item[$(3)$] A map \[ D^{\omega}\colon \{ \tau \in \Omega^1(N, \mathbb{R}) : \tau(Z^{\omega}) = 1 \}\to \Gamma(T^*N \otimes TN), \] which satisfies $\tau\circ D^{\omega}(\tau)=0$ and the transformation law given in (\ref{23082025ytu}). \end{enumerate} 
\end{quote} 
Moreover, Theorem~\ref{140825A} establishes the converse: given a lightlike manifold \((N, h, Z)\), together with a tensor field \(\mathbf{T} \in \Gamma(\Lambda^2 T^*N \otimes TN)\) and a map \(D\) as described above, subject to the corresponding transformation law, there exists a unique Cartan geometry modeled on the lightlike cone over \(N\) that recovers the original structure \((N, h, Z)\). In this reconstruction, one must determine, from the tensor \(\mathbf{T}\) and for each choice of \(\tau \in \Omega^1(N, \mathbb{R})\) such that \(\tau(Z) = 1\), a Galilean connection \(\widetilde{\nabla}^\tau\) in the Leibnizian spacetime $\big(N, \tau, \left.h\right|_{\mathrm{An}\, (\tau)}\big)$. Further details are given in Remark \ref{21072025j}. The construction of such a Galilean connection relies on the key result given in \cite[Theorem 5.27]{SB03}.

\vspace{2mm}

In the following, we outline the structure of the paper and highlight several results. Section $\ref{31102024A}$ introduces the basic facts about the future lightlike cone as homogeneous space and its relationship with the Möbius sphere. This Section, adapted from \cite{PacoLuz}, is included here for the sake of completeness. Section \ref{31102024B} recalls the fundamental notions of Cartan geometries in a general context. Our main reference is \cite{CS09}, from which we adopt the notation and several key technical results. Subsection \ref{12062025} introduces the notion of the correspondence space $\mathcal{C}(M):=\mathcal{P}/H$ associated with a Cartan geometry $(p:\mathcal{P}\rightarrow M,\omega)$ of type $(G,P)$, where $H\subset P$ is a closed subgroup. The quotient $\mathcal{C}(M)$ inherits a Cartan geometry of type $(G,H)$ whose curvature satisfies a specific annihilation property. Conversely, under suitable integrability assumptions and the same curvature condition, any Cartan geometry of type $(G,H)$ locally arises as the correspondence space of a Cartan geometry of type $(G,P)$, see \cite[Theorem 1.5.14]{CS09}. We apply this result to characterize the class of lightlike manifolds that locally appear as bundles of scales over Riemannian conformal manifolds, Remark \ref{1308202594875}.
Finally, Subsection \ref{18062025A} recalls the construction of tractor bundles associated with a Cartan geometry $(p:\mathcal{P}\rightarrow M,\omega)$ of type $(G,H)$.

\vspace{2mm}

Section \ref{13333082025A} specializes to the case where $H$ and $P$ are the subgroups of $O^+(m+1,1)$ described in Section \ref{31102024A}. Subsection \ref{1308202948} recalls the definition of a lightlike manifold $(N,h)$, where $h$ is a degenerate metric with a one–dimensional radical distribution, Definition \ref{52}. When this radical is globally spanned by a vector field $Z$, the triple $(N,h,Z)$ is considered. A set $Q$ of admissible linear frames adapted to $(N,h,Z)$ is introduced, producing a reduction of the frame bundle of $N$ to the group $H$. The Subsection concludes with the following known result, \cite[Theorem 4.4]{PacoLuz}.
\begin{quote}
Every Cartan geometry $(p:\mathcal{P}\rightarrow N,\omega)$ of type $(O^+(m+1,1),H)$ naturally defines a lightlike structure $(N,h^{\omega},Z^{\omega})$, with $\mathcal{P}$ identifiable with the bundle of admissible linear frames of this structure.
\end{quote}
Such Cartan geometries are called lightlike Cartan geometries, Definition \ref{13082937465}.

In Subsection \ref{071224C}, the framework of correspondence spaces is applied to relate lightlike Cartan geometries to Riemannian conformal geometry. Starting from a Cartan geometry $(\pi:\mathcal{G}\rightarrow M,\omega)$ of type $(O^+(m+1,1),P)$, one obtains a Riemannian conformal structure $[g]$ on $M$, \cite[Section 1.6]{CS09}. Its correspondence space $\mathcal{L}=\mathcal{G}/H$ is a lightlike manifold with lightlike metric $h^{\omega}$ and radical generator $Z^{\omega}$. An explicit description shows that $\mathcal{L}$ is exactly the fiber bundle of scales associated to the conformal class $[g]$ on $M$, and that $h^{\omega}$ pulls back to the metrics in $[g]$. The lightlike metric $h^{\omega}$ is referred to as the tautological tensor in \cite{{F-G}}. The Section shows, as a direct application of \cite[Theorem 1.5.14]{CS09}, the following characterization in  Remark \ref{1308202594875}: 
\begin{quote}
A lightlike Cartan geometry $(p:\mathcal{P}\rightarrow N,\omega)$ is locally a fiber bundle of scales of a Riemannian conformal structure if and only if $$K\left(\omega^{-1}(E), \omega^{-1}(e)\right)=0
$$ for every $e\in\mathfrak{g}_{-1}.$ The element $E$ is the grading element of the decomposition $
\mathfrak{g}=\mathfrak{g}_{-1}\oplus \mathfrak{g}_{0} \oplus \mathfrak{g}_{1}.$ 
\end{quote}

\vspace{2mm}

Section \ref{27102024y} applies the general tractor bundle construction to the standard representation of $O^{+}(m+1,1)$ on $\L^{m+2}$, leading to the notion of the standard tractor bundle for a lightlike Cartan geometry $(p:\mathcal{P} \to N, \omega)$. Concretely, the standard tractor bundle is 
$
\mathcal{T}:=\mathcal{P}\times_H \L^{m+2}\to N,
$
carrying a canonical Lorentzian bundle-like metric $\bf{h}$ induced from $\langle \,, \, \rangle$ on $\L^{m+2}$, and a canonical linear tractor connection $\nabla^{\mathcal{T}}$. By construction, $\nabla^\mathcal{T}$ is compatible with $\bf{h}$. A canonical lightlike section $\xi \in \Gamma(\mathcal{T})$ is singled out by the fact that $H$ is the isotropy subgroup of a fixed lightlike vector $\ell \in \mathcal{N}^{m+1}$. The section $\xi$ plays a central role in linking the geometry of $N$ with the tractor bundle, Proposition \ref{050125F}.
\begin{quote}
There exists a morphism of vector bundles $\Phi := \nabla^\mathcal{T}\xi:TN \to \mathcal{T}$ which is a monomorphism and an isometry between $(T_xN, h^\omega_x)$ and $(\mathcal{T}_x, {\bf{h}}_x)$ for all $x\in N$, and satisfies $\Phi(Z^\omega)=\xi$.
\end{quote}
As was mentioned, this result shows that $\mathcal{T}$ contains an isometric image of the tangent bundle $TN$, with $\xi$ corresponding to the radical generator $Z^\omega$. As a consequence, $\mathcal{T}$ can be seen as an extension of $TN$ endowed with a canonical metric connection.

To capture this structure in an abstract setting, we introduce the concept of a lightlike extension vector bundle $(\mathcal{V},{\bf g},\nabla^{\mathcal{V}},\eta)$: a rank $m+2$ real vector bundle with Lorentzian metric ${\bf g}$ and compatible connection $\nabla^{\mathcal{V}}$, together with a distinguished lightlike section $\eta$ such that $ \nabla^{\mathcal{V}}\eta$ defines a monomorphism $TN \rightarrow \mathcal{V}$, Definition \ref{16102024A}. These data induce a lightlike metric and a radical-generating vector field on $N$, recovering $(h^\omega, Z^\omega)$ in the case of a standard tractor bundle, Remark \ref{14082025ytuigh}. Finally, the Section presents the converse construction: starting from $(\mathcal{V},\bf{g},\nabla^{\mathcal{V}},\eta)$, one constructs the $O^{+}(m+1,1)$-principal fiber bundle of $\bf{g}$-orthonormal frames, reduces it to an $H$-principal fiber subbundle, and then restricts the principal connection induced by $\nabla^{\mathcal{V}}$ to obtain a lightlike Cartan connection $\omega$. The following schematic diagram provides a summary of these correspondences:
$$
			\begin{tikzpicture}
				\node (A) at (0,0) {$(p:\mathcal{P}\to N, \omega)$};
				\node (B) at (4,0) {$(\mathcal{T},\mathbf{h},\nabla^{\mathcal{T}},\xi)$};
				\node (C) at (2,-1.5) {$(N,h^{\omega},Z^{\omega})$};
				\draw[<->] (A) -- (B);
				\draw[->] (B) -- (C);
				\draw[<-] (C) -- (A);
			\end{tikzpicture}
$$
Here, the double arrow denotes a bijection. In other words, each lightlike extension vector bundle over \( N \) arises as the standard tractor bundle associated with a unique lightlike Cartan geometry (up to isomorphism). This yields a precise characterization of lightlike Cartan geometries, which will serve as a framework for the remainder of the paper.

As a conclusion to the Section, Remark \ref{1308202594875} is reformulated in terms of the curvature of the tractor connection, Remark \ref{190825A}.
\begin{quote}
     A lightlike Cartan geometry $(p:\mathcal{P}\rightarrow N,\omega)$ is locally a fiber bundle of scales of a Riemannian conformal structure if and only if $$R^{\mathcal{T}}(Z^{\omega}, V)T = 0,$$ for every \(V \in \mathfrak{X}(N)\) and \(T \in \Gamma(\mathcal{T})\), where  $R^{\mathcal{T}}$ denotes the curvature tensor of the tractor connection $\nabla^{\mathcal{T}}$.
\end{quote}

\vspace{2mm}

The starting point of Section \ref{14082025847564857} is the observation that different lightlike Cartan connections may induce the same lightlike metric and the same radical-generating vector field, since these are completely determined by the soldering form, i.e., the $\mathfrak{g}_{-1}\oplus \mathfrak{z}(\mathfrak{g}_0)$ component of the connection, see Remark \ref{281224A}. Consequently, the correspondence between lightlike Cartan geometries and the underlying lightlike structures is not injective: a Cartan connection encodes additional geometric information beyond the metric data. The purpose of this Section is to use the standard tractor bundle $\mathcal{T}$ as a tool to extract and describe this extra structure. 

Subsection \ref{27102024p} begins by noting that each $\tau\in S(N,h^{\omega},Z^{\omega}):=\left\{\tau\in\Omega^1(N,\R):\tau(Z^{\omega})=1\right\}$ determines a (screen) distribution, given by the subbundle $\mathrm{An}\,(\tau):=\{v\in TN: \tau(v)=0\}$, together with an associated field of endomorphisms $P^\tau(v):=v-\tau(v)Z^{\omega}$. The standard tractor bundle admits a $\tau$-dependent splitting
$$
\mathcal{T}\ \overset{\tau}{\cong}\ \underline{\R}\oplus \mathrm{An}\,(\tau)\oplus \underline{\R}.
$$
In this splitting, the tractor connection encodes two objects depending on $\tau$: a metric linear connection $\nabla^{\tau}$ on the Riemannian vector bundle $(\mathrm{An}\,(\tau)\rightarrow N,h^{\omega})$ and a vector bundle map $D^{\tau}:TN\to \mathrm{An}\,(\tau)$, Proposition \ref{27102024A}. The heart of the Section is the description of how $(\nabla^{\tau},D^{\tau})$ changes when $\tau$ is replaced by another $1$-form $\bar\tau\in S(N,h^{\omega},Z^{\omega})$, Lemma \ref{27102024U}.

This gives rise to one of the main results of the article, in which we reveal the underlying geometric structure of a Cartan geometry modeled on the future lightlike cone in Lorentz–Minkowski spacetime, Theorem \ref{10072025A}.
\begin{quote}
    Let $(p: \mathcal{P}\to N,\omega)$ be a lightlike Cartan geometry. Then, the manifold $N$ is endowed with the following structures:
    \begin{enumerate}
        \item[$(1)$] A lightlike metric $h^{\omega}$ and a vector field $Z^{\omega}$ spanning its radical distribution. 
        \item[$(2)$] A map $\nabla^{\omega}$ that assigns to each $\tau \in S(N,h^{\omega},Z^{\omega})$ a metric linear connection $\nabla^{\tau}$ on the vector bundle $(\mathrm{An}\, (\tau)\to N,h^{\omega})$.
        
        \item[$(3)$] A map $D^{\omega}$ that assigns to each $\tau \in S(N,h^{\omega},Z^{\omega})$ a vector bundle morphism $D^{\tau}: TN \to \mathrm{An}\, (\tau)$.
    \end{enumerate}   
    The maps $\nabla^{\omega}$ and $D^{\omega}$ must obey the transformation laws described in Lemma \ref{27102024U}.
\end{quote}
The Remarks following this Theorem highlight two conceptual payoffs. Remark \ref{21072025oki} identifies an intrinsic component of the map $\nabla^{\omega}$ determined solely by the lightlike structure $(N,h^{\omega},Z^{\omega})$, while Remark \ref{130825A} shows that the maps $\nabla^{\omega}$ and $D^{\omega}$ are completely determined by their values on a single $1$-form, as a consequence of the transformation laws of Lemma \ref{27102024U}. This makes the change equations a practical synthesis tool. We conclude this Section by translating the geometric structure exhibited in Theorem \ref{10072025A} into the language of Galilean connections, which leads to the equivalent version given in Theorem \ref{19082025hfjrfhgfujfg}.

Subsection \ref{14082025utight} develops the converse construction. Starting from a lightlike manifold $(N,h,Z)$ endowed with two maps $(\nabla,D)$ satisfying the transformation laws of Lemma \ref{27102024U}, which we shall refer to as a lightlike-compatible structure (see Definition \ref{06122024A} for details), we reconstruct the standard tractor bundle $\mathcal{T}$ by gluing the split models $\underline{\R}\oplus \mathrm{An}\,(\tau)\oplus \underline{\R}$ using the canonical transition maps $F_{\tau,\bar{\tau}}$ from Proposition \ref{17102024D}. We then define a globally well-defined tractor metric ${\bf h}$ and equip $\mathcal{T}$ with a tractor connection $\nabla^{\mathcal{T}}$. The resulting triple $(\mathcal{T},{\bf h},\nabla^{\mathcal{T}})$ comes with a canonical lightlike section $\xi$ and an embedding $\nabla^{\mathcal{T}}\xi$ of $TN$ into $\mathcal{T}$. Finally, we arrive at the main result of this work, Theorem \ref{140825A}.

\begin{quote}
Lightlike Cartan geometries $(p:\mathcal{P}\to N,\omega)$ are in bijection with lightlike manifolds $(N,h,Z)$ equipped with a lightlike-compatible structure $(\nabla,D)$. In particular, the pair $(\nabla,D)$ is exactly the additional geometric content carried by the Cartan connection beyond $(h,Z)$. The following diagram summarizes precisely this correspondence:
$$
			\begin{tikzpicture}
				\node (A) at (0,0) {$(p:\mathcal{P}\to N, \omega)$};
				\node (B) at (4,0) {$(\mathcal{T},\mathbf{h},\nabla^{\mathcal{T}},\xi)$};
				\node (C) at (2,-1.5) {$(N,h,Z,\nabla,D)$};
				\draw[<->] (A) -- (B);
				\draw[<->] (B) -- (C);
				\draw[<->] (C) -- (A);
			\end{tikzpicture}
$$
\end{quote}
To conclude, Subsection \ref{14082025mqsi} applies the constructions and results of Section \ref{14082025847564857} specifically to Sasakian manifolds and degenerate hyperplanes, offering concrete and illustrative examples that showcase the behavior and interplay of the introduced geometric structures.

\vspace{2mm}

The previous Sections established that every lightlike Cartan geometry encodes, in addition to the degenerate metric $h$ and the radical generator $Z$, a richer set of data described by a pair $(\nabla, D)$, the so–called lightlike–compatible structure. However, this extra information is not unique: different Cartan connections may give rise to the same underlying lightlike structure $(h,Z)$, while the associated pairs $(\nabla, D)$ vary accordingly. This naturally leads to the question: can one single out a canonical choice of $(\nabla, D)$ by means of intrinsic geometric conditions? Section \ref{14082025ythgirthndokgv} addresses this issue by introducing a normalization scheme for a broad subclass of lightlike Cartan geometries, which mirrors, in the lightlike setting, the well–known notion of normality in conformal Cartan geometry as stated in tractor bundles terms in \cite{CG03}.

The guiding principle is that suitable algebraic contractions of the tractor curvature should eliminate precisely the degrees of freedom left undetermined by $(h,Z)$. At each stage, one imposes collinearity or vanishing conditions on specific components of the tractor curvature, and these translate directly into uniqueness statements for the lightlike–compatible structure. Starting with a lightlike extension vector bundle $(\mathcal{T},\mathbf{h},\nabla^{\mathcal{T}},\xi)$ for $(N,h,Z)$, or equivalently with a lightlike Cartan geometry $(p: \mathcal{P}\to N, \omega)$ such that $h^{\omega}=h$ and $Z^{\omega}=Z$, the procedure is as follows: 
\begin{enumerate}
    \item The condition that \( R^{\mathcal{T}}(V,W)\xi \) be collinear with \( \xi \) for every \( V, W \in \mathfrak{X}(N) \) uniquely determines the map \(\nabla\) assigning to each \( \tau \in S(N, h, Z) \) a metric linear connection \( \nabla^{\tau} \) on the vector bundle \( (\mathrm{An}\,(\tau) \to N, h) \), as stated in Proposition~\ref{21072025y}. 
    
    However, this condition on \( R^{\mathcal{T}} \) imposes a strong restriction on the vector field \( Z \) which must necessarily be homothetic, see Remark~\ref{21072025ghj}.

    A broad and relevant class of examples is provided by properly totally umbilical lightlike hypersurfaces in Lorentzian manifolds, where $Z$ is conformal and becomes homothetic after a suitable rescaling, Remark \ref{14082025hgjdkld}.

   \item The problem of determining the map $D$ is even more delicate. In order to determine \(D(Z)\), we require that \(R^{\mathcal{T}}(Z,V)\Phi(W)\) be collinear with \(\xi\) for every \(V, W \in \mathfrak{X}(N)\). Without additional curvature assumptions, however, the existence of a solution is not guaranteed in general. Nevertheless, whenever a solution exists, it is necessarily unique. Thus, in general, this condition cannot be ensured. 
   
   An important exception occurs for lightlike manifolds locally arising as bundles of scales of Riemannian conformal structures, where the condition is automatically satisfied, see Remark \ref{190825A}.
   
   Finally, for \( m \geq 3 \), we introduce a normalization condition for the remaining part of \( D \), Proposition \ref{263473621a}, which yields a formula describing each \( D^{\tau} \) as a Schouten-like tensor. 
\end{enumerate}
Taking into account all the limitations mentioned above, we can now state Theorem \ref{21082025trg}, which guarantees, under strong restrictions, the existence of a unique lightlike Cartan connection for a certain family of lightlike manifolds. As expected, Remark \ref{010820256A} emphasizes the close relationship between these normalization conditions and those arising in Riemannian conformal geometry.

\vspace{2mm}

We conclude the Introduction by emphasizing that, despite the close relationship between conformal Cartan geometry and lightlike Cartan geometry, the structures they induce on the base manifold are fundamentally different.

Furthermore, the present analysis has been carried out in full generality. We expect that under more specific geometric assumptions—for instance, when the space of integral curves $N/Z$ of the vector field $Z$ admits the structure of a smooth manifold—some of the constructions and results developed here may acquire more explicit or geometrically transparent formulations. In such settings, it is conceivable that a broader normalization condition could be identified. These issues will be explored in future work.

\section{The future lightlike cone and the M\"{o}bius  sphere  }\label{31102024A}

\subsection{Description as Klein geometries}\label{01042025A}

\noindent All the manifolds are assumed to be smooth, Hausdorff, connected and satisfying the second axiom of countability. 
Let $\L^{m+2}$ be the $(m+2)$-dimensional Lorentz-Minkowski spacetime, that is, $\L^{m+2}$ is $\R^{m+2}$ endowed
with the Lorentzian metric 
$
\langle\, ,\, \rangle =
\sum\limits_{\scriptscriptstyle i=1}^{\scriptscriptstyle m+1}dv_{i}^{2}-dv_{m+2}^{2},
$
 where
$(v_{1},\dots, v_{m+2})$ are the canonical coordinates of
$\R^{m+2}$.  We assume $m\geq 2$, unless otherwise was stated. 
The future lightlike cone is the hypersurface given by
$$
\mathcal{N}^{m+1}=\left\{v\in \L^{m+2}:\langle v , v\rangle=0, \,\, v_{m+2}> 0\right\}.
$$
At the metric level, the future lightlike cone $\mathcal{N}^{m+1}$ inherits from the Lorentzian metric $\langle \,,\,\rangle$ a degenerate symmetric bilinear form $h$, whose radical is globally spanned by the position vector field $\mathcal{Z}\in \mathfrak{X}(\mathcal{N}^{m+1})$, defined by $\mathcal{Z}_v=(v)_{v}$, where $(\,\,)_{v}$ denotes the usual identification between $T_{v}\R^{m+2}$ and $\R^{m+2}$. To simplify notation, we shall omit the identification symbol $(\,\,)_v$ from now on. For the sake of completeness, we summarize the well-known relationships between the future 
lightlike cone and the M\"{o}bius sphere.  A detailed exposition of these facts can be found in \cite{Baum} and \cite{CS09} (see also \cite{PacoLuz}). In this Section we also fix some terminology and notations. We follow closely the treatment given in \cite{CS09}.

Let us consider $\L^{m+2}$ with a fixed basis $\mathcal{B}=(\ell,e_{1},e_{2},\dots , e_{m}, n )$, where $\ell\in \mathcal{N}^{m+1}$, and such that the corresponding matrix to $\langle \,,\,\rangle$ with respect to $\mathcal{B}$ is given by 
\begin{equation}\label{070125A}
    S:=\left(\begin{array}{ccc} 0 & 0 & 1\\ 0 & \mathrm{I}_{m} & 0 \\ 1 & 0 & 0 \end{array} \right),
\end{equation}
where $\mathrm{I}_{m}$ is the identity matrix with $m$-rows. Let $O(m+1, 1)=\{\sigma\in \mathrm{Gl}(m+2, \R):\sigma^{t}S\sigma=S\}$ be the Lorentz group, that is, the pseudo-orthogonal group of signature $(m+1,1)$. The natural action of $O(m+1,1)$ on $\L^{m+2}$ leaves the lightlike vectors invariant but does not preserve the future lightlike cone.
The orthochronous Lorentz group $O^{+}(m+1, 1)$ is the subgroup of $O(m+1, 1)$ that preserves the future lightlike cone $\mathcal{N}^{m+1}$ and is isomorphic to the 
 M\"{o}bius group, given by
$PO(m+1,1):=O(m+1, 1)/\{\pm \mathrm{Id}\}$, by means of
$$
  O^{+}(m+1, 1) \to PO(m+1,1), \quad \sigma\mapsto [\sigma],  
$$
where $[\sigma]=\{\pm \sigma\}$ is the equivalence class of $\sigma$ in $PO(m+1,1)$. Throughout this paper, we extensively use this isomorphism  to identify the orthochronous Lorentz group with the M\"{o}bius  group. Also, the Lie group $O^{+}(m+1, 1)$ acts transitively on the future lightlike cone, yielding a diffeomorphism   $\mathcal{N}^{m+1}\cong O^{+}(m+1, 1)/H,
$
where $H$ is the isotropy subgroup of the element $\ell \in \mathcal{N}^{m+1}$. The space $\mathcal{N}^{m+1}\cong O^{+}(m+1, 1)/H$ is called the $(m+1)$-dimensional lightlike homogeneous model manifold, \cite[Definition 2.1]{PacoLuz}. 

Let $\mathcal{N}^{m+1}/\sim$ be  the projective future lightlike cone, that is, the space of lines spanned by elements of $\mathcal{N}^{m+1}$ considered as a subset of the real projective space $\R P^{m+1}$.
The natural projection
$$\pi:\mathcal{N}^{m+1}\to \mathcal{N}^{m+1}/\sim$$ 
is a principal fiber bundle with structure group $\R_{>0}$. Let us consider the $m$-dimensional unit sphere $\S^{m}\subset \E^{m+1}\subset \L^{m+2}$, where $\E^{m+1}$ denotes the $(m+1)$-dimensional Euclidean space embedded in $\L^{m+2}$ at $v_{m+2}=0$. 
We have the diffeomorphism 
$$
\S^{m}\to \mathcal{N}^{m+1}/\sim, \quad x\mapsto \pi\left(
\begin{array}{c}
x \\ 1
\end{array}
\right).
$$
In this framework, the map $\pi$ reads as follows
\begin{equation*}
\mathcal{N}^{m+1} \to \S^{m}, \quad \left(v_{1},\dots, v_{m+2}\right)^t\mapsto \frac{1}{v_{m+2}}\left(v_{1},\dots, v_{m+1}\right)^t.
\end{equation*}
Under the usual identification, for every $v\in \mathcal{N}^{m+1}$, we have $T_{v}\mathcal{N}^{m+1}=v^{\perp}$, where $v^{\perp}$ denotes the orthogonal space to $v$ with respect to  $\langle\,,\,\rangle$.
The differential map of the projection $\pi$ induces an isomorphism
$
T_{v}\pi:v^{\perp}/(\R \cdot v)\to T_{\pi(v)}\S^{m}
$
which satisfies 
\begin{equation}\label{difP}
T_{v}\pi \cdot w=T_{s v}\pi \cdot sw,
\end{equation}
for every $s \in \R_{>0}$ and $w\in T_{v}\mathcal{N}^{m+1}$. Moreover, the action of the group $O^{+}(m+1, 1)$ descends to $\S^{m}\cong \pi(\mathcal{N}^{m+1})$. This action is effective and transitive,  yielding a diffeomorphism $\S^{m}\cong O^{+}(m+1,1)/P$, where $P$ is the isotropy subgroup of $\pi(\ell)$.  The group $P$ is known as the Poincar\'{e} conformal group. These descriptions as homogeneous manifolds can be summarized in the following diagram:
$$
    \begin{tikzpicture}
	
	\node (A) at (0,0) {$\mathcal{N}^{m+1}=O^{+}(m+1,1)/H$};
	\node (B) at (7,0) {$\S^{m}=O^{+}(m+1,1)/P$};
	\node (C) at (3.5,1.5) { $O^{+}(m+1,1)$};

	\draw[->] (A) -- node[midway, below] {$\pi$} (B);
	\draw[<-] (B) --  (C);
	\draw[->] (C) --  (A);
\end{tikzpicture}
$$

In a general setting, a Klein geometry is a pair $(G,H)$, where $G$ is a Lie group and $H\subset G$ is a closed subgroup. We think of the homogeneous space $G/H$ as carrying the geometric structure whose automorphisms are exactly the left actions $\ell_{g}$ by elements $g\in G$. In this terminology, we now describe the Klein geometries that correspond to the  homogeneous manifolds $\mathcal{N}^{m+1}$ and $\S^{m}$ as introduced above. The case of $\S^{m}$ is well known as the celebrated model for flat Riemannian conformal  geometry. We include this description here for completeness and due to its close relationship with the Klein geometry $\mathcal{N}^{m+1}$.

Recall that two Riemannian metrics $g$ and $g'$ on a manifold $M$ are said to be conformally related if $g'=e^{2\phi}g$ for some smooth function $\phi$ on $M$. The set of all Riemannian metrics of the form $e^{2\phi}g$ for $\phi \in C^{\infty}(M)$ is called the conformal class $[g]$ of the metric $g$. A Riemannian conformal  structure  is a pair $\left(M, [g]\right)$, where $M$ is a manifold and $[g]$ is a conformal class of Riemannian metrics on $M$.

Let $g_{_{\S^{m}}}$ be the usual round metric of constant sectional curvature $1$ on $\S^m$ and $[g_{_{\S^{m}}}]$ its conformal class. 
The conformal manifold $(\S^{m},[g_{_{\S^{m}}}])$ is called the M\"{o}bius sphere. It is well known that for $m\geq 2$, the Lie group of global conformal transformations of the M\"{o}bius  sphere $\left(\S^{m}, [g_{_{\S^{m}}}]\right)$ is the M\"{o}bius  group $PO(m+1, 1)\simeq O^{+}(m+1, 1)$. For $m\geq 3$, every conformal transformation between two connected open subsets of $\left(\S^{m},  [g_{_{\S^{m}}}]\right)$ is the restriction of a left action $\ell_{\sigma}$ by an element $\sigma\in O^{+}(m+1, 1)$. This local property does not hold for $m=2$. Finally, every diffeomorphism of the circle $\S^{1}$ gives a conformal transformation. Therefore, for $m\geq 2$,  the geometric structure that carries $\S^{m}=O^{+}(m+1,1)/P$ is exactly $\left(\S^{m},[g_{_{\S^{m}}}]\right)$. 

As a consequence of the aforementioned rigidity properties of the  M\"{o}bius  sphere, 
the following {\it Liouville type theorem} was established  in \cite[Theorem 1.1]{BFZ}.
 
 \begin{theorem}\label{isometrias}
For $m\geq 2$, the group of (global) isometries $\mathrm{Iso}(\mathcal{N}^{m+1}, h)$ is the orthochronous Lorentz group $O^{+}(m+1, 1)$.
\begin{enumerate}
\item [i)] For $m\geq 3$, every isometry between two connected open subsets of $\mathcal{N}^{m+1}$ is the restriction of the left action by an element of $O^{+}(m+1,1)$.
\item [ii)] The group of local isometries of $\mathcal{N}^{3}$ coincides with the group of local conformal transformations of $\S^{2}$.
\item [iii)] Every diffeomorphism of the circle $\S^{1}$ corresponds to an isometry of $\mathcal{N}^{2}\subset \L^{3}$.
\end{enumerate}

\end{theorem}
\noindent A slight variation on the proof of \cite[Theorem 1.1]{BFZ} permits to show that every element of $O^{+}(m+1, 1 )$ also preserves the vector field $\mathcal{Z}\in \mathfrak{X}(\mathcal{N}^{m+1})$, which spans the radical of $h$. 
That is, for every $\sigma \in O^{+}(m+1,1)$, we have $T_{v}\ell_{\sigma}\cdot \mathcal{Z}_{v}= \mathcal{Z}_{\sigma \cdot v}$, where $T_{v}\ell_{\sigma}$ denotes the differential map of the left action by $\sigma$.
Therefore, for $m\geq 2$ , the left actions by elements of the Lie group $O^{+}(m+1,1)$ coincide with the diffeomorphisms of $\mathcal{N}^{m+1}$ that preserve both $h$ and $\mathcal{Z}$. Moreover, for $m\geq 3$, every diffeomorphism between two connected open subsets of $\mathcal{N}^{m+1}$ that preserves $h$ and $\mathcal{Z}$ is the restriction of the left action by an element of $O^{+}(m+1,1)$, \cite[Corollary 2.4]{PacoLuz}.

Let us denote by $\mathrm{Sym}^{+}(TM)\to M$ the fiber bundle of positive definite symmetric bilinear forms on a manifold $M$. Thus, a Riemannian metric on $M$ is a section of this fiber bundle. Note that for any section $s:\S^{m}\to \mathcal{N}^{m+1}$ of $\pi$, given by $s(x)=e^{\phi(x)}(x^t,1)^t$ for some $\phi\in C^{\infty}(\S^{m})$,  we have $s^{*}h=e^{2\phi}g_{_{\S^{m}}} \in [g_{_{\S^{m}}}]$. Conversely,
for every metric $g\in [g_{_{\S^{m}}}]$, there exists a section $s$ of $\pi$ with $g=s^{*}h$. Following \cite{BZ17}, the future lightlike cone $\mathcal{N}^{m+1}$ can be seen as a submanifold $$
 c\colon \mathcal{N}^{m+1} \to \mathrm{Sym}^{+}(T\S^{m}), 
$$ in the following way.  The bilinear form $h$ induces  a (positive definite) inner product on the quotient space $v^{\perp}/(\R\cdot v)$. Assuming that $\pi(v)=x$, the element $c(v)\in \mathrm{Sym}^{+}(T_{x}\S^{m})$ is defined by requiring that the map $T_{v}\pi: \left(v^{\perp}/(\R\cdot v),h\right)\to \left(T_{x}\S^{m},c(v)\right)$ is an isometry. If we replace $v$ by $sv$ for some $s\in \R_{>0}$ and use (\ref{difP}), the carried bilinear form on $T_{x}\S^{m}$ is scaled by a factor of $s^2$. Hence, we obtain an inner product on $T_{x}\S^{m}$, determined up to positive scalar multiples. In particular, for every $x\in \S^{m}$,
 $$c\left(\pi^{-1}(x)\right)=\left\{\left.s^{2}g_{_{\S^{m}}}\right|_x: s>0\right\},$$
which is the positive ray spanned by $g_{_{\S^{m}}}$ at $x\in \S^{m}$. Consequently, $\mathcal{N}^{m+1}$ is the fiber bundle of scales associated with the conformal structure of the  M\"{o}bius sphere $\left(\S^{m},[g_{_{\S^{m}}}]\right)$. This fact will be important in Subsection \ref{071224C}.

\subsection{Algebraic description of $\mathcal{N}^{m+1}$ and the Möbius sphere}\label{13082025yt}

In this Subsection, we recall the description of the homogeneous manifolds 
$\mathcal{N}^{m+1}=O^{+}(m+1,1)/H$ and $\S^{m}=O^{+}(m+1,1)/P$ at the level of Lie algebras.
First, recall that the elements of the Poincar\'{e} conformal group $P$ correspond to the elements of $O^{+}(m+1,1)$ of the form
$$
P=\left\{\left(\begin{array}{ccc} \lambda & -\lambda w^{t}g & -\frac{\lambda}{2}\| w\|^2 \\ 0 & g& w \\ 0 & 0 & \lambda^{-1} \end{array} \right): \lambda \in \R_{>0} ,\,\, w\in \R^{m},\,\, g\in O(m)\right\},
$$
where $\|w\|^2=g_{_{\R^{m}}}(w,w)$ denotes the square of the usual Euclidean norm on $\R^{m}$ \cite[Proposition 1.6.3]{CS09}. The Lie group $H\subset O^{+}(m+1,1)$ corresponds to the elements of $P$ with $\lambda= 1$. 
Therefore, $H$ is a closed normal subgroup of $P$, with $P/H\cong \R_{>0}$.

\begin{remark}\label{71224A}
    {\rm As pointed out in \cite{PacoLuz}, the group of rigid motions $ \mathrm{Euc}(\E^{m})=\R^{m}\rtimes O(m) $ of the $m$-dimensional Euclidean space $\E^m$ is  isomorphic to $H$  via the map
\begin{equation}\label{iso1}
\mathrm{Euc}(\E^{m}) \ni\left(\begin{array}{cc} 1 &  0  \\ w & g  \end{array} \right)
\mapsto \left(\begin{array}{ccc} 1 & -  w^{t}g & -\frac{1}{2}\| w\|^2\\ 0 & g & w \\ 0 & 0 & 1 \end{array} \right)\in H.
\end{equation}
 }
\end{remark}
At the level of Lie algebras, the Lie algebra $\mathfrak{g}$ of $O^{+}(m+1,1)$ is given by
$$
\mathfrak{g}= \left\{\left(\begin{array}{ccc} a & Z & 0\\ X & A & -Z^{t} \\ 0 & -X^{t} & -a \end{array} \right): a\in \R, X\in \R^{m}, Z\in (\R^{m})^{*}, A\in \mathfrak{o}(m)\right\}.
$$
The decomposition
$
\mathfrak{g}=\mathfrak{g}_{-1}\oplus \mathfrak{g}_{0} \oplus \mathfrak{g}_{1}
$ where
$$
\mathfrak{g}_{-1}=\left\{\left(\begin{array}{ccc} 0 & 0 & 0\\ X & 0 & 0 \\ 0 & -X^{t} & 0 \end{array} \right)\right\}, 
\mathfrak{g}_{0}=\left\{\left(\begin{array}{ccc} a & 0 & 0\\ 0 & A & 0 \\ 0 & 0 & -a \end{array} \right)\right\}, 
\mathfrak{g}_{1}=\left\{\left(\begin{array}{ccc} 0 & Z & 0\\ 0 & 0 & -Z^{t} \\ 0 & 0 & 0 \end{array} \right)\right\}$$
defines a grading of $\mathfrak{g}$, that is, $[\mathfrak{g}_{i},\mathfrak{g}_{j}]\subset \mathfrak{g}_{i+j}$, where $\mathfrak{g}_{k}=0$ when $k\notin \{-1,0,1\}$.
The Lie algebra of $P$ is given by $\mathfrak{p}=\mathfrak{g}_{0}\oplus \mathfrak{g}_{1}$, and the Lie algebra $\mathfrak{h}$ of $H$ is the Lie subalgebra of $\mathfrak{p}$ determined by $a=0$, \cite[Section 1.6.3]{CS09}. We have $\mathfrak{g}_{0}=\mathfrak{z}(\mathfrak{g}_{0})\oplus [\mathfrak{g}_{0}, \mathfrak{g}_{0}]$, where $\mathfrak{z}(\mathfrak{g}_{0})$ is the center of the Lie algebra $\mathfrak{g}_{0}$, and $\mathfrak{h}=[\mathfrak{g}_{0}, \mathfrak{g}_{0}]\oplus \mathfrak{g}_{1}$. The quotient vector space $\mathfrak{g}/ \mathfrak{h}$ can be identified with $\R \times \R^{m}$ as follows: 
\begin{equation}\label{151}
\left(\begin{array}{ccc} a & Z & 0\\ X & A & -Z^{t} \\ 0 & -X^{t} & -a \end{array} \right)+ \mathfrak{h} \mapsto \left(
\begin{array}{c}
a \\ X
\end{array}
\right)\in \R \times \R^{m}.
\end{equation}
Recall that the quotient adjoint representation $\underline{\mathrm{Ad}}:H\longrightarrow \mathrm{Gl}(\mathfrak{g}/\mathfrak{h)}$  is defined as 
\begin{equation*}
 \underline{\mathrm{Ad}}(h):\mathfrak{g}/\mathfrak{h} \to \mathfrak{g}/\mathfrak{h},\quad Y+\mathfrak{h}\mapsto \mathrm{Ad}(h)(Y)+\mathfrak{h},\end{equation*}
where $\mathrm{Ad}$ denotes the adjoint representation of $O^{+}(m+1,1)$. 

We conclude this Section with several results on the Klein geometry $\left(O^{+}(m+1,1), H\right)$ for  $m\geq 2$, as given in \cite[Lemma 2.5]{PacoLuz}.
\begin{enumerate}
    \item The Lie algebra $\mathfrak{h}$ does not admit any reductive complement in $\mathfrak{g}$.
\item Taking into account (\ref{151}), the quotient adjoint representation is given by
\begin{equation}\label{071224B}
    \underline{\mathrm{Ad}}(h)\left(\begin{array}{c}
a \\ X
\end{array}
\right)=\left(\begin{array}{c}
a-  g_{_{\R^{m}}}( w, gX) \\  gX
\end{array}
\right),
\end{equation}
where $h$ corresponds to $\left(\begin{array}{cc} 1 &  0  \\ w & g  \end{array} \right)$ under the isomorphism given in $(\ref{iso1})$.
\item The map $\underline{\mathrm{Ad}}$ is injective. That is, the Klein geometry $\left(O^{+}(m+1,1), H\right)$ is a first-order Klein geometry.
\end{enumerate}

\section{Cartan geometries}\label{31102024B}

For the sake of completeness, we include an introduction to Cartan geometries in a fully general context in this Section. We will focus only on the part of the theory relevant to this article and will closely follow  \cite[Section 1.5]{CS09}.

\subsection{Definition and properties}\label{130820295y}
 
For a principal fiber bundle $p:\mathcal{P}\to M$ with structure group $H$, let us denote by $r^{h}$ the right translation on $\mathcal{P}$ by $h \in H$, and by $\zeta_{X}\in \mathfrak{X}(\mathcal{P})$ the fundamental vector field corresponding to $X\in \mathfrak{h}=\mathrm{Lie}(H)$, given by
$
\zeta_{X}(u):=\left.\frac{d}{dt}\right|_{0}\left(u \cdot \mathrm{exp}\left(tX\right)\right).
$

Let $G$ be a Lie group and $H\subset G$ a closed subgroup, and let $\mathfrak{g}$ be the Lie algebra of $G$. A Cartan geometry of type $(G,H)$ on a manifold $M$ consists of:
\begin{enumerate}
\item A principal fiber bundle $p:\mathcal{P}\to M$ with structure group $H$.

\item A $\mathfrak{g}$-valued $1$-form $\omega \in \Omega^{1}(\mathcal{P}, \mathfrak{g})$, called the Cartan connection, such that for every $u\in \mathcal{P}, h\in H$, and $X\in \mathfrak{h}$, the following conditions hold:
\begin{enumerate}
\item $
\omega(u):T_{u}\mathcal{P}\to \mathfrak{g}
$
is a linear isomorphism.
\item  $(r^{h})^{*}(\omega)=\mathrm{Ad}(h^{-1})\circ \omega$.
\item $\omega(u)(\zeta_{X}(u))=X.$
 \end{enumerate} 
\end{enumerate}
For every $X\in \mathfrak{g}$, the constant vector field $\omega^{-1}(X)\in \mathfrak{X}(\mathcal{P})$ is defined by the condition $\omega(u) \left(\omega^{-1}\left(X\right)(u)\right)=X$ for all $u\in \mathcal{P}$. The curvature form of a Cartan geometry is given by $K:=d\omega+\frac{1}{2}[\omega, \omega]\in \Omega^{2}(\mathcal{P}, \mathfrak{g})$. Equivalently, we have
$$
K(\xi ,\eta)=d\omega (\xi, \eta)+[\omega(\xi), \omega (\eta)], \quad \xi, \eta\in \mathfrak{X}(\mathcal{P}). 
$$
A Cartan connection $\omega$ is said to be torsion-free when $K$
takes values in $\mathfrak{h}$.
All the information of $K$ is contained in the curvature function $\kappa : \mathcal{P}\to \Lambda^{2}\mathfrak{g}^{*}\otimes \mathfrak{g}$, defined as $$\kappa(u)(X,Y)=K\left(\omega^{-1}(X)(u),\omega^{-1}(Y)(u)\right).$$ Since $K$ is horizontal, meaning that it vanishes when evaluated on a vertical tangent vector, the curvature function may be viewed as $\kappa : \mathcal{P}\to \Lambda^{2}(\mathfrak{g}/\mathfrak{h})^{*}\otimes \mathfrak{g}$, see \cite[Lemma 1.5.1]{CS09}.

The canonical projection $p:G\to G/H$, endowed with the (left) Maurer-Cartan form $\omega_{G}\in \Omega^{1}(G, \mathfrak{g})$,  is called the homogeneous model for Cartan geometries of type $(G,H)$. The Maurer-Cartan equation implies that the homogeneous model of any Cartan geometry has zero curvature \cite[Section 1.2.4]{CS09}.

A Cartan connection provides a description of the tangent bundle of the base manifold $M$ as the associated fiber bundle $\mathcal{P}\times_{{H}} (\mathfrak{g}/\mathfrak{h})$ for the quotient adjoint representation $\underline{\mathrm{Ad}}:H\longrightarrow \mathrm{Gl}(\mathfrak{g}/\mathfrak{h)}$  given by
\begin{equation*}
\underline{\mathrm{Ad}}(h):\mathfrak{g}/\mathfrak{h} \to \mathfrak{g}/\mathfrak{h},\quad Y+\mathfrak{h}\mapsto \mathrm{Ad}(h)(Y)+\mathfrak{h},\end{equation*}
where $\mathrm{Ad}$ denotes the adjoint representation of $G$. Explicitly, for each $u\in \mathcal{P}$ with $p(u)=x\in M$, there is a canonical linear isomorphism 
$\phi_{u}:T_{x}M \to \mathfrak{g}/\mathfrak{h}$ such that the following diagram commutes:
\begin{equation}\label{isomor}
\begin{CD}
T_{u}\mathcal{P}  @>\omega(u)>> \mathfrak{g}\\
@V T_{u} p  VV @VV \mathrm{proj} V \\
T_{x}M @>>\phi_{u}  > \mathfrak{g}/\mathfrak{h}
\end{CD}\quad \quad \text{with } \phi_{uh}=\underline{\mathrm{Ad}}(h^{-1})\circ\phi_{u} \text{ for all }h\in H.
\end{equation}
Then, the canonical isomorphism of vector bundles over $M$ is given by
\begin{equation}\label{142}
TM \cong \mathcal{P}\times_{{H}} (\mathfrak{g}/\mathfrak{h}),\quad (x,v)\in TM\mapsto [u, \phi_{u}(v)],
\end{equation}
where $p(u)=x$, see \cite[Theorem 3.15]{Sharpe}.

An isomorphism of Cartan geometries $(p:\mathcal{P}\to M, \omega)$ and $(p':\mathcal{P}'\to M', \omega')$ with the same model $(G,H)$ is a principal fiber bundle isomorphism $(F,f)$ 
$$
 \begin{CD}
\mathcal{P}  @>F >> \mathcal{P}'\\
@V p VV @VV p' V \\
M @>>  f > M'
\end{CD}
$$
such that $F^{*}(\omega ')=\omega$.
That is, $F$ is a diffeomorphism from $\mathcal{P}$ to $\mathcal{P}'$ (and so is $f$) such that $F\circ r^{h}=r^{h}\circ F$ for all $h\in H$, and $F^{*}(\omega ')=\omega$. 
In the particular case that both Cartan geometries are the same, an isomorphism is called an automorphism.
The group $\mathrm{Aut}(\mathcal{P}, \omega)$ of all automorphisms of the Cartan geometry $(p:\mathcal{P}\to M, \omega)$ over a connected manifold $M$ is a Lie group (possibly with uncountably many connected
components) and has dimension at most $\mathrm{dim}(G)$, \cite[Theorem 1.5.11]{CS09}.

A Cartan geometry $p:\mathcal{P}\to M$ of type $(G,H)$ has curvature form $K=0$ if and only if every point $x\in M$ has an open neighborhood $U\subset M$ such that $\left(p:p^{-1}(U)\to U, \omega\vert_{U}\right)$ is isomorphic to the restriction of the homogeneous model $\left(G\to G/H, \omega_{G}\right)$ to an open neighborhood of $o:=eH$,  \cite[Proposition 1.5.2]{CS09}.

\subsection{Correspondence spaces}\label{12062025}

This Subsection is based on \cite[Section 1.5.13]{CS09}. Let $(\pi:\mathcal{P}\to M, \omega)$ be a Cartan geometry of type $(G,P)$, and let $H\subset P$ be a closed subgroup. The correspondence space for $H\subset P$ is given by $\mathcal{C}(M):=\mathcal{P}/H$. 
The natural projection $f:\mathcal{C}(M)\to M$ defines a fiber bundle over $M$ with fiber given by the homogeneous space $P/H$,
and $(p\colon \mathcal{P}\to \mathcal{C}(M), \omega)$ is a Cartan geometry of type $(G,H)$ on $\mathcal{C}(M)$, see \cite[Proposition 1.5.13]{CS09}. The following commutative diagram holds:

$$
\begin{tikzpicture}
	
	\node (A) at (0,0) {$\mathcal{C}(M)$};
	\node (B) at (3,0) {$M$};
	\node (C) at (1.5,1.5) { $\mathcal{P}$};

	\draw[->] (A) -- node[midway, below] {$f$} (B);
	\draw[<-] (B) -- node[midway, right] {$\pi$} (C);
	\draw[->] (C) -- node[midway, left] {$p$} (A);
\end{tikzpicture}
$$
From (\ref{142}), we obtain a description of the tangent bundle of $\mathcal{C}(M)$ as  $\mathcal{P}\times _{{H}}(\mathfrak{g}/ \mathfrak{h})$. In these terms, the vertical distribution of $f$ corresponds to $\mathcal{P}\times _{{H}}(\mathfrak{p}/ \mathfrak{h})\subset \mathcal{P}\times _{{H}}(\mathfrak{g}/ \mathfrak{h})=T\mathcal{C}(M)$, \cite[Proposition 1.5.13]{CS09}.
The corresponding curvature function
$
k^{\mathcal{C}(M)}:\mathcal{P}\to \Lambda^{2}(\mathfrak{g}/\mathfrak{h})^{*}\otimes \mathfrak{g}
$
satisfies 
\begin{equation}\label{081224A}
  k^{\mathcal{C}(M)}(u)(X+\mathfrak{h}, \cdot)=0,   
\end{equation}
for every $X\in \mathfrak{p}$ and $u\in \mathcal{P}$.

On the other hand, the converse can also be characterized. Let $G$ be a Lie group and $H\subset P\subset G$ closed subgroups. Consider a Cartan geometry $(p:\mathcal{P}\to N, \omega)$ of type $(G,H)$ such that the distribution $\mathcal{V}(N):= \mathcal{P}\times _{H}(\mathfrak{p}/ \mathfrak{h}) \subset TN$ is integrable. 
A (local) twistor space for $N$ is a smooth manifold $M$ together with an open subset $U\subset N$ and a surjective submersion $f:U\to M$ such that $\mathcal{V}_{x}N=\mathrm{Ker}(T_{x}f)$, i.e., a (local) leaf space for the foliation defined by $\mathcal{V}(N),$ \cite[Definition 1.5.14]{CS09}. 

The curvature property (\ref{081224A}) locally characterizes the correspondence spaces as follows, \cite[Theorem 1.5.14]{CS09}.
Let $(p:\mathcal{P}\to N, \omega)$ be a Cartan geometry of type $(G,H)$ such that the distribution $\mathcal{V}(N)= \mathcal{P}\times_{H}(\mathfrak{p}/ \mathfrak{h}) \subset TN$ is integrable, and assume that the curvature function 
$
k^{N}:\mathcal{P}\to \Lambda^{2}(\mathfrak{g}/\mathfrak{h})^{*}\otimes \mathfrak{g}
$
satisfies $k^{N}(u)(X+\mathfrak{h}, \cdot)=0$ for every $X\in \mathfrak{p}$ and $u\in\mathcal{P}$. Then, for any sufficiently small local twistor space $f:U\to M$, there exists a Cartan geometry of type $(G,P)$ on $M$ such that the restricted Cartan geometry $\left(p:p ^{-1}(U) \to U, \left.\omega\right|_{U}\right)$ of type $(G,H)$ is isomorphic to an open subset of the correspondence space $\mathcal{C}(M)$. In the case where $P/H$ is connected, this Cartan geometry on $M$ is unique.

\subsection{Tractor bundles}\label{18062025A}

Given  a Cartan geometry $(p:\mathcal{P}\to M,\omega)$ of type $(G,H)$, we can construct the extended associated fiber bundle
$\bar{\mathcal{P}}:=\mathcal{P}\times_{H}G$, where  the left action of $H$ on $G$ is given by multiplication.
Now, the natural projection $\pi \colon \bar{\mathcal{P}} \to M$ is a principal fiber bundle with structure group $G$, and
there is a map 
$$
j\colon \mathcal{P}\to \bar{\mathcal{P}}, \quad u\mapsto [u,e],
$$
where, as usual, $[u,e]$ denotes the equivalence class of $(u,e)$. Note that $\pi \circ j=p.$

As a particular case of \cite[Theorem 1.5.6]{CS09}, there exists a unique principal connection $\gamma \in \Omega^{1}(\bar{\mathcal{P}}, \mathfrak{g})$ such that $j^{*}\gamma= \omega.$ Therefore, for every representation $\rho \colon G\to \mathrm{GL}(\mathbb{V})$ of the group $G$ on a vector space $\mathbb{V}$, we obtain a linear connection $\nabla^{\gamma}$ on the associated vector bundle $\bar{\mathcal{P}}\times_{G}\mathbb{V}\to M$. Recall that $\nabla^{\gamma}$ is given as follows. Let $\sigma$ be a smooth section of $\bar{\mathcal{P}}\times_{G}\mathbb{V}$, and let $U\subset M$ be an open subset. Then, for every $x\in U$, we have
\begin{equation}\label{13102023A}
	\nabla^\gamma_{W_{x}}\sigma:=\left[s(x),\,W_{x}(f)+\rho'\left(\gamma(s(x))(T_{x}s\cdot W_{x})\right)(f(x))\right],
\end{equation}
where $W\in \mathfrak{X}(M)$ and $\sigma|_U = [s,\, f]$ for a local section $s:U\subset M\rightarrow \bar{\mathcal{P}}$ and a smooth function $f\in\mathcal{C}^{\infty}(U,\mathbb{V})$. We denote by $\rho'\colon \mathfrak{g}\to \mathfrak{gl}(\mathbb{V})$ the derivative of the representation $\rho.$ 

The restriction of $\rho$ to the subgroup $H$ also defines a representation, and  we have a natural vector bundle isomorphism
$$
\mathcal{P} \times_H \mathbb{V}\cong  \bar{\mathcal{P}}\times_G \mathbb{V},\quad [u,v]\mapsto [j(u), v].
$$
In this setting,  $\mathcal{P} \times_H \mathbb{V}$ is called the tractor bundle for the representation $\rho$ and, via the above isomorphism, it carries a canonical linear connection known as the tractor connection. Moreover, any $G$-invariant scalar product $\langle \,,\,\rangle_{\mathbb{V}}$ on $\mathbb{V}$ induces a well-defined bundle-like metric $\mathbf{h}$ on the vector bundle $\mathcal{P} \times_H \mathbb{V}$ given by
$$
	\mathbf{h}\left([u,\,v_1],[u,\,v_2]\right):=\langle v_1,v_2\rangle_{\mathbb{V}},
$$
for all $[u,\,v_1],[u,\,v_2]\in \mathcal{P} \times_H \mathbb{V}$. By definition, the tractor connection is a metric linear connection with respect to any metric $\mathbf{h}$ of this form.

\section{Cartan geometries with model the future lightlike cone}\label{13333082025A}
Now, we return to the particular case where $H$ and $P$ are the Lie groups considered in Section \ref{31102024A}.
\subsection{Lightlike Cartan geometries and the approach to induced structures}\label{1308202948}
We recall the following definition (see \cite[Definition 4.1]{PacoLuz}).
\begin{definition}\label{52}
A lightlike manifold is a pair $(N,h)$, where $N$ is a smooth manifold of dimension $m+1$ and $h$ is a lightlike metric on $N$. That is, $h$ is a symmetric $(0,2)$-tensor field on $N$ satisfying the following conditions:
\begin{enumerate}
\item $h(v,v)\geq 0$ for all $v\in T_{x}N$. 
\item The radical distribution $\mathrm{Rad}(h_{x})=\{v\in T_{x}N: h(v,\cdot)=0\}$ is a one-dimensional distribution on $N$.
\end{enumerate}
If there exists a vector field $Z\in \mathfrak{X}(N)$ that globally spans the radical distribution $\mathrm{Rad}(h)$, we say that $\mathrm{Rad}(h)$ is orientable,  and we write $(N,h,Z)$ to specify such a vector field $Z$.
\end{definition}

Every lightlike metric $h$ induces a bundle-like Riemannian metric $\overline{h}$ on the vector bundle $\mathcal{E}:=TN/\mathrm{Rad}(h)$, defined by $\overline{h}\big([V],[W]\big):=h(V,W)$, for all $V,W\in\mathfrak{X}(N)$, where the brackets denote equivalence classes in the quotient space $\mathcal{E}$. Given a triple $(N,h,Z)$ as above, the Lie derivative $\mathcal{L}_{Z}h$ induces a symmetric bilinear map $\mathcal{L}_{Z} \overline{h}$ on $\mathcal{E}$ via $\big(\mathcal{L}_{Z} \overline{h}\big)\big([V],[W]\big):=\big(\mathcal{L}_{Z}h\big)(V,W).$ Thus, the bundle-like metric $\overline{h}$ determines a self-adjoint vector bundle endomorphism $A_{Z}:\mathcal{E}\rightarrow \mathcal{E}$, defined by the condition 
\begin{equation}\label{16072025D}
    \big(\mathcal{L}_{Z} \overline{h}\big)([V],[W])=2\overline{h}\big(A_{Z}[V],[W]\big),
\end{equation}
for all $V,W\in\mathfrak{X}(N)$. Following \cite{BZ17}, the lightlike manifold $(N,h,Z)$ is said to be generic if $A_{Z}$ is an isomorphism on $\mathcal{E}$. This condition is independent of the choice of the vector field $Z$ spanning $\mathrm{Rad}(h)$. For simplicity, we will also denote the metric $\overline{h}$ by $h$, unless this may lead to confusion.

For a given lightlike manifold $(N,h,Z)$, the set $\mathcal{Q}\subset \mathcal{P}^{1}N$ of all admissible linear frames  is defined as
\begin{equation}\label{laQ}
\mathcal{Q}=\big\{b=(Z_x,e_{1}, \dots , e_{m})\in \mathcal{P}^{1}_{x}(N): x\in N, \,\, h(e_{i}, e_{j})=\delta_{ij}\big\}.
\end{equation}
Thus, $\mathcal{Q}$ defines a reduction of the structure group to 
$H$ via the Lie group monomorphism
$$
H\to \mathrm{GL}(m+1,\R), \quad \left(\begin{array}{ccc} 1 & -  w^{t}g & -\frac{1}{2}\| w\|^2\\ 0 & g & w \\ 0 & 0 & 1 \end{array} \right)\mapsto \left(\begin{array}{cc} 1 & -  w^{t}g  \\ 0 & g  \end{array} \right).
$$
Consequently, we have
$$
b\cdot h=\left(Z_x, \sum_{i=1}^{m}g_{i1}(-w_{i}Z_x+e_{i}),\dots , \sum_{i=1}^{m}g_{im}(-w_{i}Z_x+e_{i})\right).
$$

The quotient vector space $\mathfrak{g}/ \mathfrak{h}\simeq\R \times \R^{m}$, as given in (\ref{151}), is endowed with the lightlike metric
\begin{equation}\label{20062025A}
q\left(\left(
\begin{array}{c}
a \\ X
\end{array}
\right), \left(
\begin{array}{c}
b \\ Y
\end{array}
\right)\right)=g_{_{\R^{m}}}( X , Y),
\end{equation}
where the vector  $(1,0)^{t}\in \mathfrak{g}/ \mathfrak{h}$ spans the radical distribution $\mathrm{Rad}(q)$. From (\ref{071224B}), it follows that both $q$ and $(1,0)^{t}$ are $H$-invariant under the representation $\underline{\mathrm{Ad}}(H)$. Now, suppose that  $(p:\mathcal{P}\to N, \omega)$ is a Cartan geometry of type $\left(O^{+}(m+1,1),H\right)$. Taking into account the isomorphism of vector bundles $TN\cong \mathcal{P}\times_{H}(\mathfrak{g}/\mathfrak{h})$ given in $(\ref{142})$, one can show that $\omega$
induces a lightlike metric $h^{\omega}$ on the base manifold $N$, as well as a vector field $Z^{\omega}\in \mathfrak{X}(N)$ that globally spans the distribution $\mathrm{Rad}(h^{\omega})$, see \cite[Theorem 4.4]{PacoLuz}. 

\begin{remark}\label{281224A}
    {\rm \begin{enumerate}
        \item Let $(p:\mathcal{P}\to N, \omega)$ be a Cartan geometry of type $\left(O^{+}(m+1,1),H\right)$, and let $Q^{\omega}$ be the set of admissible linear frames given in (\ref{laQ}) for the lightlike manifold $(N,h^{\omega}, Z^{\omega})$. 
Since $(O^{+}(m+1,1),H)$ is a first-order Klein geometry, the total space $\mathcal{P}$ can be identified with the set of all admissible linear frames $Q^{\omega}$ via the following principal fiber bundle isomorphism over $N$, see \cite[Exercise 3.21]{Sharpe}, 
$$
\mathcal{S}:\mathcal{P}\to Q^{\omega},\quad u\mapsto \left(Z^{\omega}_{p(u)}, \phi_{u}^{-1} \left(
\begin{array}{c}
0 \\ \bar{e}_{1}
\end{array}
\right), \dots, \phi_{u}^{-1} \left(
\begin{array}{c}
0 \\ \bar{e}_{m}
\end{array}
\right)\right),
$$
where  $\left((1,0)^{t}, (0,\bar{e}_{1}^t)^{t}, \dots ,(0,\bar{e}_{m}^t)^{t}\right)$ is the fixed basis for 
$\mathfrak{g}/\mathfrak{h}\simeq \R \times \R^{m}$, with $(\bar{e}_{1}, \dots, \bar{e}_{m})$ being the canonical basis of $\R^m$. 
From now on, we identify $\mathcal{P}$ with $\mathcal{Q}^{\omega}$.

\item Recall the soldering form $\theta\in \Omega^{1}(\mathcal{P}^{1}N, \R^{m+1})$, defined by $\theta (b)(\xi)=b^{-1}(T_{b}\pi \cdot \xi)$, where $\pi:\mathcal{P}^{1}N\rightarrow N$, $b\in\mathcal{P}^{1}N$, and $\xi \in T_{b}\mathcal{P}^{1}N $. 
The soldering form $\theta$ is naturally induced on $Q^{\omega}$, and from $(\ref{151})$ and diagram $(\ref{isomor})$ we obtain
$$
\left.\theta\right|_{Q^{\omega}}(b)(\xi)=b^{-1}(T_{b}p \cdot \xi)=\mathrm{proj}\circ\omega(b)(\xi),
$$
where $\xi \in T_{b}Q^{\omega}$, and $\mathrm{proj}: \mathfrak{g}\to \mathfrak{g}/\mathfrak{h}$ denotes the natural projection. Therefore, the Cartan connection $\omega$ can be written as 
$$
\omega=\left.\theta\right|_{Q^{\omega}}\oplus \omega_{[\mathfrak{g}_{0}, \mathfrak{g}_{0}]}\oplus \omega_{1},
$$
where the subscripts denote the projections of $\mathfrak{g}$ onto $[\mathfrak{g}_{0}, \mathfrak{g}_{0}]$ and $\mathfrak{g}_{1}$, respectively, within the Cartan connection $\omega$. In particular, we obtain
$$
    \omega_{-1}\oplus\omega_{\mathfrak{z}(\mathfrak{g}_{0})}=\left.\theta\right|_{Q^{\omega}}.
$$

\item 
The vector field $Z^{\omega}\in \mathfrak{X}(N)$ satisfies 
$Z^{\omega}_{p(u)}=T_{u}p\cdot \omega^{-1}(E)(u)$ for every $u\in \mathcal{P},$
where $$E={\left(\begin{array}{ccc} 1 &  0 & 0\\ 0 & 0 & 0 \\ 0 & 0 & -1 \end{array} \right)}\in \mathfrak{z}(\mathfrak{g}_{0})$$ is the grading element of the decomposition $
\mathfrak{g}=\mathfrak{g}_{-1}\oplus \mathfrak{g}_{0} \oplus \mathfrak{g}_{1}
$. That is, the splitting $\mathfrak{g}=\mathfrak{g}_{-1}\oplus \mathfrak{g}_{0} \oplus \mathfrak{g}_{1}$ corresponds to the decomposition into eigenspaces of the adjoint action $\mathrm{ad}(E)=[E, *]$, with eigenvalues $-1, 0,1$, respectively. 
\end{enumerate}
}
\end{remark}

\begin{definition}\label{13082937465}
    A Cartan geometry $(p:\mathcal{P}\to N, \omega)$ modeled on $(O^{+}(m+1,1),H)$ is called a lightlike Cartan geometry on $N$, and its Cartan connection $\omega$ is referred to as a lightlike Cartan connection.
\end{definition}

\subsection{The correspondence  space for $H\subset P\subset O^{+}(m+1,1)$}\label{071224C}

In this Subsection, we aim to relate the classical approach to Riemannian conformal geometry via Cartan connections with lightlike Cartan geometries, using the perspective of correspondence spaces. This viewpoint provides a useful framework for comparing our construction with the classical theory. Once again, we follow \cite[Section 1.6]{CS09}.

Let $(\pi:\mathcal{G}\to M, \omega)$ be a Cartan geometry of type $(O^{+}(m+1,1), P)$. As is well known, considering that the quotient adjoint representation for every $\sigma=\tiny{\left(\begin{array}{ccc} \lambda & -\lambda w^{t}g & -\frac{\lambda}{2}\| w\|^2 \\ 0 & g& w \\ 0 & 0 & \lambda^{-1} \end{array} \right)}\in P$ is given by
$$
\underline{\mathrm{Ad}}(\sigma):\mathfrak{g}/\mathfrak{p} \to \mathfrak{g}/\mathfrak{p},\quad X+\mathfrak{p}\mapsto \lambda^{-1} gX+\mathfrak{p},
$$
 where $X\in \mathfrak{g}_{-1}$ and $TM\cong\mathcal{G}\times_{P}(\mathfrak{g}/ \mathfrak{p})$, this Cartan connection $\omega\in \Omega^{1}(\mathcal{P}, \mathfrak{g})$ induces a conformal structure $[g]$ on $M$. Moreover, there is a equivalence of categories between conformal structures and normal Cartan geometries of type $(O^{+}(m+1,1),P);$ see, for instance, \cite[Theorem 1.6.7]{CS09} for details.

Let $\mathcal{L}= \mathcal{G}/H$ be the correspondence space for $H\subset P$, with the natural projection $$f\colon \mathcal{L}\to M,$$ and consider 
the Cartan geometry $(p:\mathcal{G}\to \mathcal{L}, \omega)$ of type $(O^{+}(m+1,1), H)$. We know that $\omega$ determines a lightlike metric $h^{\omega}$ and a vector field $Z^{\omega}$, which spans the radical of $h^{\omega}$ on $\mathcal{L}$. As a consequence of Remark $\ref{281224A}$, the total space $\mathcal{G}$ coincides with the set of admissible linear frames $\mathcal{Q}^{\omega}$ of $(N,h^{\omega},Z^{\omega})$. 
Note that the bilinear form $h^{\omega}$ induces  a (positive definite) inner product on the quotient space $T_{y}\mathcal{L}/(\R\cdot Z^{\omega}_y)$. Similarly to the case of the future lightlike cone discussed in Subsection \ref{01042025A}, there is a natural map
$$
c\colon \mathcal{L}\to \mathrm{Sym}^{+}(TM), \quad y \mapsto c(y),
$$
 where $f(y)=x$, and $c(y)\in \mathrm{Sym}^{+}(T_{x}M)$ is defined by requiring that $$T_{y}f: \left(T_{y}\mathcal{L}/(\R\cdot Z^{\omega}_y), h^{\omega}\right)\to \left(T_{x}M, c(y)\right)$$ is an isometry. From $(\ref{isomor})$, we can consider the following commutative diagram:
\begin{center}
\begin{tikzcd} 
T_{u}\mathcal{G}\arrow[r,"\omega(u)"] \arrow[d, swap, "T_{u}p"]& \mathfrak{g}\arrow[d,"\textrm{proj}_{\mathfrak{g}_{-1}\oplus \mathfrak{z}(\mathfrak{g}_{0})}"] \\
T_{y}\mathcal{L} \arrow[r,  "\phi_{u}"{pos=0.5},shorten >= -5pt, shorten <= 0pt] \arrow[d,  swap, "T_{y}f"]& \,\,\,\mathfrak{g}_{-1}\oplus \mathfrak{z}(\mathfrak{g}_{0}) \arrow[d,  "\textrm{proj}_{\mathfrak{g}_{-1}}"]\\
T_{x}M \arrow[r, "\psi_{u}"] & \mathfrak{g}_{-1}
\end{tikzcd}
\end{center}
One easily sees that  the values of $c(y)$ with $f(y)=x$ cover the entire positive ray formed by the metrics in the conformal class $[g]$ at the point $x\in M$. That is, the correspondence space $\mathcal{L}= \mathcal{G}/H$ for $H\subset P$ coincides with the fiber bundle of scales associated with  the conformal structure induced on $M$ by the Cartan connection $\omega$, \cite[Section 1.6.5]{CS09}. 

The lightlike metric $h^{\omega}$ induced by $\omega$ on $\mathcal{L}$ is given by 
$$
h^{\omega}(\xi, \eta)=c(y)\left(T_{y}f\cdot \xi,T_{y}f\cdot \eta\right),
$$
for all $\xi, \eta \in T_{y}\mathcal{L}$. The lightlike metric $h^{\omega}$ is referred to as the tautological tensor in \cite{{F-G}}.
Regarding the vector field $Z^{\omega}$, recall that for any $u\in\mathcal{G}$ satisfying $p(u)=y\in \mathcal{L}$, we have
$$
Z^{\omega}_y=T_{u}p\cdot \omega^{-1}(E)(u).
$$
Since $E\in \mathfrak{p}$, it follows that $\omega^{-1}(E)=\zeta_{E}$, and consequently, the flow of $\omega^{-1}(E)$ is given by $\mathrm{Fl}^{\omega^{-1}(E)}_{t}(u)=u\cdot \mathrm{exp}(tE)$ for every $u\in \mathcal{G}$. Thus, the integral curve of $Z^{\omega}$ starting at $y$ is 
$$
\mathrm{Fl}^{Z^{\omega}}_{t}(y)=p\Big(u\cdot \mathrm{exp}(tE)\Big).
$$
Note that the curve $\mathrm{Fl}^{\omega^{-1}(E)}_{t}(u)$ is vertical with respect to the projection $\pi=f\circ p$, but not with respect to $p$.

Let $p_{0}: \mathcal{G}_{0}\to M$ be the $CO(m)$-principal fiber bundle describing the conformal structure on $M$ induced by $\omega\in \Omega^{1}(\mathcal{G}, \mathfrak{g})$. Let $\mathcal{Q}^{\omega}$ be the total space of admissible linear frames of the tautological metric $h^{\omega}$ on $\mathcal{L}$. There is a natural map
  $$
  \mathcal{F}\colon \mathcal{Q}^{\omega} \to \mathcal{G}_{0}, \quad (Z^{\omega}_y, e_{1},\dots, e_{m})\mapsto u_0:=(T_{y}f\cdot e_{1}, \dots , T_{y}f\cdot e_{m}).
  $$

Recall that the total space $\mathcal{G},$ where the normal Cartan connection $\omega$ is defined, is constructed from $\mathcal{G}_{0}$ as follows, \cite[Section 1.6.4]{CS09}.
  For every $u_0\in \mathcal{G}_{0}$, we define $\mathcal{G}_{u_0}$ as the set of linear isomorphisms $\gamma=\gamma_{-1}+\gamma_0\colon T_{u_0}\mathcal{G}_{0}\to \mathfrak{g}_{-1}\oplus \mathfrak{g}_{0}$ satisfying $\gamma_{-1}=\left.\theta\right|_{\mathcal{G}_0}(u_0),\,\,\gamma_{0}(\zeta_{(a, A)}(u_0))=(a, A)$ for every $(a,A)\in \mathfrak{g}_0\simeq\mathfrak{co}(m)$ 
  and
  $$
  \left.d\theta\right|_{\mathcal{G}_0}(u_0)(\xi, \eta)+[\gamma_{0}(\xi), \gamma_{-1}(\eta)]+[\gamma_{-1}(\xi), \gamma_{0}(\eta)]=0, \quad \xi, \eta \in T_{u_0}\mathcal{G}_{0}.
  $$
  Then, $\mathcal{G}$ is the disjoint union of the $\mathcal{G}_{u_0}$ for every $u_0\in\mathcal{G}_0$. The natural  projection
  $
  \Pi\colon \mathcal{G}\to \mathcal{G}_{0}
  $
satisfies $\pi=p_{0}\circ \Pi$ and defines a principal fiber bundle with structure group $P_{+}=\{\sigma \in  P: \lambda=1, g=\mathrm{Id}_{m} \}\simeq \R^{m}$. Finally,
the component $\omega_{-1}$ of the Cartan connection $\omega \in \Omega^{1}(\mathcal{G}, \mathfrak{g})$ is the pullback of the soldering form $\theta$ on $\mathcal{G}_{0}$, i.e., $\omega_{-1}=\Pi^{*}\left(\left.\theta\right|_{\mathcal{G}_0}\right)$, see \cite[Proposition 1.6.4]{CS09}.

Additionally, we have a natural map
$$
q\colon \mathcal{G}_{0}\to \mathcal{L}, \quad u_0 \mapsto q(u_0),
$$
where $q(u_0)$ denotes the inner product in $T_{p_{0}(u_0)}M$, with $u_0$ being an orthonormal basis  with respect to $q(u_0)$.

All manifolds involved here are quotient spaces of $\mathcal{G}$, as follows:
$$M=\mathcal{G}/P,\quad  \mathcal{G}_{0}= \mathcal{G}/P_{+} \quad \textrm{ and }\quad \mathcal{L}= \mathcal{G}/H.$$  
Taking into account that $P_{+}\subset H \subset P\subset O^{+}(m+1,1)$, the maps above are the natural projections, and, in particular, $q$
 is an $H/P_{+}\simeq O(m)$-principal fiber bundle. We can summarize the above maps in the following commutative diagram:

\[
\begin{tikzcd}
& \mathcal{Q}^\omega \arrow[dr, " \mathcal{F}"] & \\
\mathcal{G} \arrow[ur, "\mathcal{S}"] \arrow[rr, "\Pi"] \arrow[d, "p"'] & & \mathcal{G}_0 \arrow[d, "p_0"] \\
\mathcal{L} \arrow[rr, "f"'] & & M
\end{tikzcd}
\]

The natural question of when a lightlike Cartan geometry is locally the fiber bundle of scales associated with a Riemannian conformal structure can be addressed by applying the characterization of correspondence spaces, as follows.

Let 
$(p:\mathcal{P}\to N, \omega)$ be a lightlike Cartan geometry on $N$. The distribution $\mathcal{V}(N):= \mathcal{P}\times _{H}(\mathfrak{p}/ \mathfrak{h}) \subset TN$ is one-dimensional and therefore integrable. Moreover, $P/H\simeq \R_{>0}$ is connected. Assume the curvature function 
$
k^{N}:\mathcal{P}\to \Lambda^{2}(\mathfrak{g}/\mathfrak{h})^{*}\otimes \mathfrak{g}
$
satisfies 
\begin{equation}\label{050125A}
 k^{N}(b)(E+\mathfrak{h}, \cdot)=0,   
\end{equation}
for every $b\in\mathcal{P}$.  Therefore, as discussed in Subsection \ref{12062025}, a direct application of \cite[Theorem 1.5.14]{CS09} shows that for any sufficiently small open subset $U\subset N$, the orbit space $U/Z^{\omega}$ 
admits a unique Cartan geometry of type $(O^{+}(m+1,1),P)$ such that the restricted Cartan geometry $(p:p ^{-1}(U) \to U, \omega\vert_{U})$ is isomorphic 
to the fiber bundle of scales of $(p^{-1}(U)\to U/Z^{\omega}, \omega\vert_{U}).$ 

\begin{remark}\label{1308202594875}
    {\rm Let $(p:\mathcal{P}\to N, \omega)$ be a lightlike Cartan geometry on $N$.  Its curvature $K$ is horizontal, that is, $K\left(\zeta_{X}, \cdot\right)=0$  for every $X\in \mathfrak{h}$.  Therefore, the curvature condition $(\ref{050125A})$ is equivalent to
$$
K\left(\omega^{-1}(E), \omega^{-1}(e)\right)=0
$$
  for every $e\in\mathfrak{g}_{-1}.$   Taking into account that $E$ is the grading element of the decomposition $
\mathfrak{g}=\mathfrak{g}_{-1}\oplus \mathfrak{g}_{0} \oplus \mathfrak{g}_{1}$, we conclude that (\ref{050125A})
is equivalent to
$$
\omega\left([\omega^{-1}(E), \omega^{-1}(e)]\right)=-e, \quad e\in\mathfrak{g}_{-1}.
$$
    }
\end{remark}

\section{The standard tractor bundle of a lightlike Cartan geometry}\label{27102024y}

Let $(p:\mathcal{P}\rightarrow N,\omega)$ be a lightlike Cartan geometry. We now particularize the general construction of tractor bundles given in Subsection \ref{18062025A} to the standard representation of 
$O^{+}(m+1, 1)$ on $\L^{m+2}$ by linear isometries.

\begin{definition}\label{23102024A}
	 Let $\mathcal{T}:=\mathcal{P} \times_H \mathbb{L}^{m+2}\rightarrow N$ be the tractor bundle for the standard representation of $O^{+}(m+1, 1)$.  The vector bundle $\mathcal{T}$ is called the standard tractor bundle of the lightlike Cartan geometry $(p:\mathcal{P}\rightarrow N,\omega)$.  We write $\nabla^\mathcal{T}$ for the canonical linear connection on $\mathcal{T}\to N$.
\end{definition}

Therefore, by construction, the vector bundle $\mathcal{T}$ carries a canonical bundle-like metric $\mathbf{h}$ of Lorentzian signature given by
	$$
	\mathbf{h}([b,\,v_{1}], [b,\,v_{2}])=\langle v_{1}, v_{2}\rangle,
	$$
where $\langle \,,\, \rangle$ denotes the Lorentzian metric on $\mathbb{L}^{m+2}$ and $[b,\,v_{1}], [b,\, v_{2}]\in \mathcal{T}.$ Recall that the linear connection $\nabla^{\mathcal{T}}$ is metric with respect to $\bf{h}$.

\begin{proposition}\label{050125F}
    Let $(p:\mathcal{P}\rightarrow N,\omega)$ be a lightlike Cartan geometry with standard tractor bundle $\mathcal{T}\to N$. Then, there exists a distinguished lightlike section $\xi\in\Gamma(\mathcal{T})$ such that the following morphism of vector bundles over $N$,
        $$
        \Phi:=\nabla^{\mathcal{T}}\xi\colon TN\longrightarrow \mathcal{T}, \quad W\mapsto \nabla^{\mathcal{T}}_{W}\xi,
        $$ is a monomorphism and an isometry between $\left(T_x N,h^{\omega}_x\right)$ and $\left(\mathcal{T}_x,\mathbf{h}_x\right)$, for every $x\in N$. Moreover, $\Phi(Z^{\omega})=\xi$. 
    
\end{proposition}
\begin{proof}
 
Taking into account that $H$ is the isotropy subgroup of the element $\ell \in \mathcal{N}^{m+1}$, we obtain a well-defined distinguished section $\xi\in\Gamma(\mathcal{T})$ given by 
$$
\xi_x:=[b,\,\ell]\in\mathcal{T}_x,
$$
where $b\in\mathcal{P}$ with $p(b)=x$. It is clear that
$
\mathbf{h}\left(\xi,\xi\right)=\langle\ell,\ell\rangle=0.
$

For every open subset $U\subset N$, we have $\xi|_U = [s,\, \ell]$ for some local smooth section $s\colon U\to \mathcal{P}$. Therefore, for every $W\in \mathfrak{X}(N)$, a direct computation using (\ref{13102023A}) yields, at each point $x\in U$,
\begin{equation}\label{050125D}
   \nabla^\mathcal{T}_{W_{x}}\xi=\left[s(x),\,\omega(s(x))(T_{x}s\cdot W_{x})(\ell)\right], 
\end{equation}
where we have used that $\rho$ is the standard representation of $O^{+}(m+1,1)$. Now, assume that $\nabla^\mathcal{T}_{W_{x}}\xi=\left[s(x),0\right]$, that is, \begin{equation}\label{050125B}
\omega(s(x))(T_{x}s\cdot W_{x})(\ell)=0.
\end{equation}
Recall that $\omega(s(x))(T_{x}s\cdot W_{x})\in\mathfrak{g}$. Thus, it can be written in the form
$$
\omega(s(x))(T_x s\cdot W_{x})=\left(\begin{array}{ccc} a & Z & 0\\ X & A & -Z^{t} \\ 0 & -X^{t} & -a \end{array} \right)\in\mathfrak{g}.
$$
Consequently, Equation ($\ref{050125B}$) implies that $a=0$ and $X=0$.
From the definition of a Cartan connection, it follows that there exists a unique $Y\in\mathfrak{h}$ such that $T_x s\cdot W_{x}=\zeta_{Y}\left(s(x)\right)$, where $\zeta_Y$ is the fundamental vector field corresponding to $Y\in\mathfrak{h}$. Since $s$ is a section, we have $T_{s(x)}p\cdot T_x s\cdot W_{x}=W_{x}$. On the other hand, we also know that $T_{s(x)}p\cdot \zeta_{Y}\left(s(x)\right)=0$. It follows that $W_{x}=0$. Therefore, $\Phi=\nabla^{\mathcal{T}}\xi$ is a monomorphism of vector bundles over $N$.
 
From $(\ref{20062025A})$ and $(\ref{050125D})$, and taking into account \cite[Theorem 4.4]{PacoLuz}, we obtain that for all $V,W\in \mathfrak{X}(N)$ and $x\in N$,
 \begin{align*}
  \mathbf{h}_x\left(\nabla^{\mathcal{T}}_{V_{x}}\xi,\nabla^{\mathcal{T}}_{W_{x}}\xi\right) & = \Big\langle \omega(s(x))(T_{x}s\cdot V_{x})(\ell) , \omega(s(x))(T_{x}s\cdot W_{x})(\ell)\Big\rangle \\
  & =q\Big( \theta(s(x))(T_{x}s\cdot V_{x}) , \theta(s(x))(T_{x}s\cdot W_{x})\Big)= h^{\omega}_x(V_{x}, W_{x}),
\end{align*}
which shows that $\Phi$ is an isometry. Finally, from the definition of $Z^{\omega}$ given in \cite[Theorem 4.4]{PacoLuz}, a straightforward computation yields $\nabla^{\mathcal{T}}_{Z^{\omega}}\xi= \xi$.
\end{proof}

\begin{remark}\label{220825A}
{\rm As a consequence of the previous Proposition, the standard tractor bundle $\mathcal{T}\to N$ can be viewed as an extension of the tangent bundle $TN$. Moreover, $\mathcal{T}$ is naturally equipped with a canonical metric linear connection $\nabla^{\mathcal{T}}$. This stands in sharp contrast to the situation on lightlike manifolds, which, in general, do not admit a natural metric connection.
}
\end{remark}

Starting from a lightlike Cartan geometry $(p:\mathcal{P}\to N,\omega)$,  we have constructed the standard tractor bundle $(\mathcal{T},\mathbf{h},\nabla^{\mathcal{T}},\xi)$. The next natural step is to investigate the converse problem: given a quadruple $(\mathcal{T},\mathbf{h},\nabla^{\mathcal{T}},\xi)$ satisfying the properties of Proposition \ref{050125F}, we aim to reconstruct a lightlike Cartan geometry $(p:\mathcal{P}\to N, \omega)$, and show that this procedure effectively inverts the original construction. To this end, we introduce an abstract formulation of the vector bundle given in Definition \ref{23102024A} and Proposition \ref{050125F} (see Definition \ref{16102024A} below). This will allow us to derive an equivalent characterization of lightlike Cartan geometries in terms of standard tractor bundles. A perspective that is well established for various types of Cartan geometries. We include a brief outline of this construction here for the sake of completeness.

\begin{definition}\label{16102024A} 
Let $N$ be an $(m+1)$-dimensional manifold. A lightlike extension vector bundle over $N$ consists of the following data:
\begin{enumerate}
    \item A rank $m + 2$ real vector bundle $\mathcal{V}\rightarrow N$ endowed with a bundle-like metric $\mathbf{g}$ of Lorentzian signature and a linear connection $\nabla^{\mathcal{V}}$ such that 
		$\nabla^{\mathcal{V}} \mathbf{g}=0$. 
    \item A distinguished lightlike section $\eta \in\Gamma(\mathcal{V})$ such that the map 
		$$
			\Psi(w):=\nabla_w^{\mathcal{V}}\eta,\,\text{where }x\in N\,\text{and }w\in T_x N,
		$$
        defines a monomorphism of vector bundles over $N$,
		$$
			\begin{tikzpicture}
				
				\node (A) at (0,0) {$TN$};
				\node (B) at (2,0) {$\mathcal{V}$};
				\node (C) at (1,-1) {$N$};
				\node (D) at  (1,0.3) {$\Psi$};
				
				\draw[->] (A) -- (B);
				\draw[->] (B) -- (C);
				\draw[<-] (C) -- (A);
			\end{tikzpicture}
		$$
		
\end{enumerate}

\end{definition}
In this terminology, Definition \ref{23102024A} and Proposition \ref{050125F} constructs a lightlike extension vector bundle over $N$ from a lightlike Cartan geometry $(p:\mathcal{P}\to N, \omega)$.

Let us note that a lightlike extension vector bundle $(\mathcal{V}, \bf{g}, \nabla^{\mathcal{V}}, \eta)$ endows the manifold $N$ with a lightlike metric $g$ and a vector field $Y\in \mathfrak{X}(N)$ that spans its radical. Specifically, we define
\begin{equation*}
g(V,W):=\mathbf{g}\left(\Psi(V),\Psi(W)\right),
\end{equation*}
for all $V,W\in\mathfrak{X}(N)$. Since $\nabla^{\mathcal{V}}$ is metric with respect to $\bf{g}$, it follows that $\mathbf{g}(\Psi(V), \eta)=0$.
As $\Psi$ is a vector bundle monomorphism, it follows from a dimensional argument that $\Psi(T_x N)=\eta_x^{\perp}$ for every $x\in N$. Therefore, we have $g=\left.\mathbf{g}\right|_{\eta^{\perp}}$, and hence $g$ is a lightlike metric on $N$.
Moreover, there exists a unique lightlike vector field $Y\in\mathfrak{X}(N)$ such that $\Psi\left(Y\right)=\eta$. Since $\eta$ globally spans the radical of $\left.\mathbf{g}\right|_{\xi^{\perp}}$, it follows that $Y$ globally spans the radical of $g$. 

\begin{remark}\label{14082025ytuigh}
{\rm Let $(p:\mathcal{P}\to N,\omega)$ be a lightlike Cartan geometry, and let $(\mathcal{T},\mathbf{h},\nabla^{\mathcal{T}},\xi)$ be its associated standard tractor bundle. Then, it is clear that the lightlike metric and the vector field spanning its radical, as induced by the standard tractor bundle, coincide with $h^{\omega}$ and $Z^{\omega}$, respectively.}
\end{remark}

As previously mentioned, we now outline the converse construction. Starting from a lightlike extension vector bundle $(\mathcal{V}, \bf{g}, \nabla^{\mathcal{V}}, \eta)$,  we construct a lightlike Cartan geometry $(p:\mathcal{P}\to N, \omega)$ such that its associated standard tractor bundle coincides with $(\mathcal{V}, \bf{g}, \nabla^{\mathcal{V}}, \eta)$. To this end, we consider the $O^{+}(m+1,1)$-principal fiber bundle of $\mathbf{g}$-orthonormal frames of $\mathcal{V}\to N$, defined by
 $$
 \bar{\mathcal{P}}:=\left\{u\colon \L^{m+2}\to \left(\mathcal{V}_{x},\mathbf{g}_x\right): x\in N \textrm{ and } u \textrm{ is an isometry with } \mathbf{g}_x(u(\ell)-u(n),\eta_{x})<0 \right\}.
 $$
 Recall that $\L^{m+2}$ is equipped with a fixed basis $\mathcal{B}=(\ell,e_{1},e_{2},\dots , e_{m}, n )$, where $\ell\in \mathcal{N}^{m+1}$, and such that the corresponding matrix to $\langle \,,\,\rangle$ is given by (\ref{070125A}). In other words, $ \bar{\mathcal{P}}$ consists of the 
 $\mathbf{g}$-orthonormal frames that preserve the time orientation of $\L^{m+2}$ determined by the timelike vector $\ell-n$.  Thus, $\mathcal{V}$ is the associated vector bundle  $ \bar{\mathcal{P}}\times_{O^{+}(m+1,1)}\L^{m+2}$.

 In order to obtain the desired Cartan connection, we restrict 
$\bar{\mathcal{P}}$ to a principal fiber subbundle $\mathcal{P}$,
  which will serve as the underlying principal fiber bundle of the Cartan geometry. We define $\mathcal{P}$ by 
 $$
 \mathcal{P}:=\left\{u \in \bar{\mathcal{P}}: u(\ell)=\eta_{x}\right\}.
 $$
It is straightforward to verify that the natural inclusion $j\colon \mathcal{P}\hookrightarrow \bar{\mathcal{P}} $ defines a reduction of the structure group from $O^{+}(m+1, 1)$ to $H$. Moreover, we have the identification $\mathcal{P}\times_{H}O^{+}(m+1,1)\cong \bar{\mathcal{P}}.$ Note that $\mathcal{P}$ can be interpreted as the principal fiber bundle of admissible linear frames for the lightlike metric $g$ on $N$, via the map
$$
\mathcal{P}\to \mathcal{Q}, \quad u \longmapsto \left(Y_x,\Psi ^{-1}(u(e_{1})), \dots , \Psi^{-1}(u(e_{m}))\right).
$$
The inverse map is well-defined. Given an admissible linear frame $\left(Y_x,w_{1}, \dots , w_{m}\right) \in \mathcal{Q}$, there exists a unique isometry $u\colon \L^{m+2}\to\left(\mathcal{V}_{x},\mathbf{g}_x\right)$ such that $u(\ell)= \eta_{x}$, $u(e_{j})=\Psi(w_{j})$ for $j=1,\dots, m,$ and  $\textbf{g}_x(u(n), \eta_{x})=1$.

Observe that, in the identification $\mathcal{V}\cong \mathcal{P}\times_{H}\L^{m+2}$, the distinguished lightlike section $\eta \in\Gamma(\mathcal{V})$
 is given at each point $x\in N$ by
 $$
 \eta_{x}=[u, \ell]
 $$
 for any $u\in \mathcal{P}$ over $x.$

To conclude, there exists a unique principal connection $\gamma\in\Omega^{1}\big(\bar{\mathcal{P}},\mathfrak{g}\big)$ whose associated linear connection is precisely $\nabla^{\mathcal{V}}$. Let us consider $\omega:=j^{*}\gamma \in \Omega^{1}(\mathcal{P}, \mathfrak{g})$. Then the pair $(p: \mathcal{P}\to N, \omega)$ defines a lightlike Cartan geometry. Indeed, properties (b) and (c) of $\omega$ follow directly from the fact that $\gamma$ is a principal connection and $j\colon \mathcal{P}\hookrightarrow \bar{\mathcal{P}}$ is a reduction of the structure group from $O^{+}(m+1, 1)$ to $H$. Property (a) follows from a computation analogous to that in Equation $(\ref{050125D})$ of Proposition \ref{050125F}, taking into account that $\Psi$ is a monomorphism.

Finally, it is clear that this construction yields the inverse procedure to that used for obtaining a lightlike extension vector bundle over $N$ from a lightlike Cartan geometry $(p: \mathcal{P}\to N, \omega)$. In particular, every lightlike extension vector bundle arises, in an injective manner, as the standard tractor bundle associated with a lightlike Cartan geometry. This observation justifies the adoption of a unified notation for lightlike extension vector bundles and standard tractor bundles. Consequently, we obtain an equivalent description of lightlike Cartan geometries in terms of such vector bundles. The following diagram summarizes the results established in this Section.

$$
			\begin{tikzpicture}
				
				\node (A) at (0,0) {$(p:\mathcal{P}\to N, \omega)$};
				\node (B) at (4,0) {$(\mathcal{T},\mathbf{h},\nabla^{\mathcal{T}},\xi)$};
				\node (C) at (2,-1.5) {$(N,h^{\omega},Z^{\omega})$};
				
				\draw[<->] (A) -- (B);
				\draw[->] (B) -- (C);
				\draw[<-] (C) -- (A);
			\end{tikzpicture}
$$

\begin{remark}
    {\rm In order to analyze the elements underlying a lightlike Cartan geometry $(p\colon \mathcal{P}\to N,\omega)$, we begin by observing that
    $$\mathbf{h}(R^{\mathcal{T}}(V,W)\xi,\xi)=0,$$ where $R^{\mathcal{T}}$ denotes the curvature tensor\footnote{Our sign convention for the curvature tensor of a linear connection is given by 
$$
R^{\mathcal{T}}(V,W)T=\nabla^{\mathcal{T}}_V \nabla^{\mathcal{T}}_W T- \nabla^{\mathcal{T}}_W \nabla^{\mathcal{T}}_V T- \nabla^{\mathcal{T}}_{[V,W]}T,
$$ 
for all $V,W\in\mathfrak{X}(N)$ and $T\in\Gamma(\mathcal{T})$.} of the tractor connection $\nabla^{\mathcal{T}}$. Consequently, the term $R^{\mathcal{T}}(V,W)\xi$  defines a tensor $\mathbf{T}^{\omega}\in \Gamma(\Lambda^{2}T^{*}N\otimes TN)$ through the relation
\begin{equation}\label{040825A}
    R^{\mathcal{T}}(V,W)\xi=\Phi(\mathbf{T}^{\omega}(V,W)).
\end{equation}}
\end{remark}

\begin{remark}\label{190825A}
    {\rm In Remark $\ref{1308202594875}$,  lightlike Cartan geometries \((p: \mathcal{P} \to N, \omega)\) that are locally the bundle of scales associated with a Riemannian conformal structure were characterized in terms of their curvature by the condition 
$
K(\omega^{-1}(E), \cdot) = 0,
$
where \(E\) is the grading element in the decomposition 
$
\mathfrak{g} = \mathfrak{g}_{-1} \oplus \mathfrak{g}_0 \oplus \mathfrak{g}_1.
$
By \cite[Proposition 1.3.4(5)]{CS09}, this curvature condition can be equivalently reformulated as a property of the curvature of the associated standard tractor bundle \((\mathcal{T}, \mathbf{h}, \nabla^{\mathcal{T}}, \xi)\), namely:
\[
R^{\mathcal{T}}(Z^{\omega}, V)T = 0,
\]
for every \(V \in \mathfrak{X}(N)\) and \(T \in \Gamma(\mathcal{T})\).}
\end{remark}

\section{Identification of additional geometric structures via standard tractor bundles}\label{14082025847564857}
As shown in Remark \ref{281224A}, different lightlike Cartan connections may induce the same lightlike metric and the same radical-generating vector field, since both are completely determined by the soldering form, that is, by the component of the connection taking values in $\mathfrak{g}_{-1}\oplus \mathfrak{z}(\mathfrak{g}_{0})$. This reveals that the correspondence between lightlike Cartan geometries and the underlying lightlike structures is not injective. In particular, a lightlike Cartan connection encodes additional geometric information that is not captured by the metric data alone. To extract and analyze this additional structure, we will use standard tractor bundles as a fundamental tool.

\subsection{Decomposition of the standard tractor bundle}\label{27102024p}

Let $(N, h,Z)$ be an $(m+1)$-dimensional lightlike manifold. Consider the subbundle of the cotangent bundle $T^{*}N$ defined by
    $$\mathbb{B}:=\left\{\tau_{x}\in T^{*}_{x}N: x\in N,\,\tau_{x}(Z_{x})=1 \right\}.
    $$
    The set of sections of $\mathbb{B}$,
    $$
    S(N,h,Z):=\left\{\tau\in\Omega^1(N,\R):\tau(Z)=1\right\},
    $$
    is called the bundle of screen distributions associated to $(N,h,Z)$. A standard partition of unity argument guarantees that $ S(N,h,Z)\neq \emptyset$.
Each $\tau\in S(N,h,Z)$ determines a (screen) distribution given by the subbundle $\mathrm{An}\,(\tau):=\{v\in TN: \tau(v)=0\}$ and an associated field of endomorphisms
\begin{equation}\label{13112024}
P^{\tau}(v)=v-\tau(v)Z,\quad v\in TN.
\end{equation}
Note that $h$ induces a positive definite bundle-like metric on $\mathrm{An}\,(\tau)\to N$. By a slight abuse of notation, we shall denote the induced metric on $\mathrm{An}\,(\tau)$ also by $h$. We also have the following decomposition of the tangent bundle,
\begin{equation}\label{02072025A}
TN\overset{\tau}{= } \underline{\R}\oplus \mathrm{An}\,(\tau),\quad v\mapsto \left(\tau(v), P^{\tau}(v)\right),
\end{equation}
where $\underline{\R}:=N\times\R\rightarrow N$ denotes the trivial bundle. 

\begin{remark}
{\rm A Leibnizian spacetime is a triple $(M, \Omega, g )$, where $M$ is a manifold, $\Omega$ is a nowhere-vanishing $1$-form, and $g$ is a positive definite bundle-like metric defined on the vector bundle $\mathrm{An}\,(\Omega)\to M$, see details in \cite{SB03}. In our framework, each choice of $\tau \in S(N,h,Z)$ determines a Leibnizian spacetime $(N,\tau,h)$. The map $P^{\tau}$ is referred to as the spacelike projection along $Z$.} 
\end{remark}

Assume now that we have a standard tractor bundle  $(\mathcal{T}, \mathbf{h},\nabla^{\mathcal{T}}, \xi)$ constructed from a lightlike Cartan geometry $(p: \mathcal{P}\to N,\omega)$. For each $\tau \in S(N, h^{\omega}, Z^{\omega})$, there exists a unique lightlike section $\eta^{\tau}\in \Gamma(\mathcal{T})$ determined by the conditions
$$
\textbf{h}(\xi, \eta^{\tau})=1, \quad \textbf{h}(\Phi(\mathrm{An}\,(\tau)), \eta^{\tau})=0,
$$
where $\Phi$ is the map given in Proposition $\ref{050125F}$. This yields a $\tau$-dependent decomposition of the standard tractor bundle given by
\begin{equation}\label{01072025A}
\mathcal{T}\overset{\tau}{\longrightarrow} \underline{\R}\oplus \mathrm{An}\,(\tau)\oplus \underline{\R},\quad T_{x}\mapsto T^{\tau}_{x}:=\begin{pmatrix}\mathbf{h}_x(T_{x}, \eta^{\tau}_{x}) \\ X_x \\  \mathbf{h}_x(T_{x}, \xi_{x})\end{pmatrix},
\end{equation}
where $X_x\in \mathrm{An}\, (\tau_{x})$ and $T_{x}=\mathbf{h}_x(T_{x}, \eta^{\tau}_{x})\, \xi_{x}+ \Phi(X_x)+ \mathbf{h}_x(T_{x}, \xi_{x})\,\eta^{\tau}_{x}\, \in \mathcal{T}_{x}$.
Accordingly, for any $T\in\Gamma\left(\mathcal{T}\right)$, its expression in the decomposition associated with $\tau$ is given by
$$
	T^{\tau}=\begin{pmatrix}\alpha \\ X \\  \beta\end{pmatrix},\,\,\text{with } \alpha,\beta\in\mathcal{C}^{\infty}(N,\R) \,\,\text{and}\,\,X\in\Gamma(\mathrm{An}\,(\tau)).
	$$
In this decomposition, the bundle-like metric $\mathbf{h}$ is given by
	$$
	\mathbf{h}\left(\begin{pmatrix} \alpha_{1} \\ X_1 \\  \beta_1\end{pmatrix},\begin{pmatrix}\alpha_2 \\ X_2 \\  \beta_2\end{pmatrix}\right)=\alpha_1\beta_2+\beta_1\alpha_2+h^{\omega}\big(X_1,X_2\big).
	$$
Note that for every $W\in\mathfrak{X}(N)$, the $1$-form $\tau$ can be recovered via the identity
\begin{equation}\label{100125A}
   \tau(W)=\mathbf{h}(\Phi(W),\eta^{\tau}). 
\end{equation}
Given two $1$-forms $\tau,\overline{\tau}\in S(N,h^{\omega},Z^{\omega})$, observe that  $\mathbf{h}(\xi,\eta^{\tau}-\eta^{\overline{\tau}})=0 $, which implies that $\eta^{\tau}-\eta^{\overline{\tau}}\in \Phi(TN)$. We therefore define the difference vector field between $\tau$ and $\bar{\tau}$ by 
$$
K_{\tau, \bar{\tau}}:=\Phi^{-1}(\eta^{\tau}- \eta^{\bar{\tau}})\in \mathfrak{X}(N).
$$
Let $(E_{1}, \dots, E_{m})$ be a local frame of $\mathrm{An}\, (\bar{\tau})$ with $h^{\omega}(E_{i},E_{j})=\delta_{ij}$. We define the vector field
\begin{equation}\label{30072893845}
   L_{\tau, \bar{\tau}}:=\sum_{i=1}^{m}\tau(E_{i})E_{i}.
\end{equation}
Although it is expressed in terms of a local frame, $L_{\tau, \bar{\tau}}$ is globally well-defined, meaning that its definition is independent of the choice of orthonormal frame and thus determines a smooth vector field on the whole manifold. We record here two elementary properties of these vector fields:
\begin{equation}\label{090825D}
L_{\tau,\overline{\tau}}+L_{\overline{\tau},\tau}=h^{\omega}(L_{\tau,\overline{\tau}}, L_{\tau,\overline{\tau}})Z^{\omega}\quad \textrm{ and }\quad L_{\alpha, \tau}+L_{ \tau, \overline{\tau}}=L_{\alpha ,\overline{\tau}}+h^{\omega}(L_{\alpha, \tau}, L_{\tau,\overline{\tau}})Z^{\omega},
\end{equation}
for every $\alpha,\tau, \overline{\tau} \in S(N,h^{\omega},Z^{\omega}).$

\vspace{2mm}

A natural question is how the decompositions of $\mathcal{T}$ associated with two different 1-forms $\tau, \bar{\tau} \in S(N, h^{\omega}, Z^{\omega})$ are related.
To address this, we establish the following results.
\begin{lemma}\label{27102024D}
Let $(\mathcal{T}, \mathbf{h},\nabla^{\mathcal{T}}, \xi)$ be the standard tractor bundle constructed from a lightlike Cartan geometry $(p: \mathcal{P}\to N,\omega)$. Given two $1$-forms $\tau,\overline{\tau}\in S(N,h^{\omega},Z^{\omega})$ and $W\in\mathfrak{X}(N)$, the following identities hold:
\begin{enumerate}
\item[$(1)$] $\mathbf{h}(\eta^{\tau},\eta^{\overline{\tau}})=-\tau\left(K_{\tau, \bar{\tau}}\right)=\overline{\tau}\left(K_{\tau, \bar{\tau}}\right)$.
\item[$(2)$] $h^{\omega}\big(W,K_{\tau, \bar{\tau}}\big)=\tau(W)-\overline{\tau}(W)$.
\item[$(3)$]  $K_{\tau, \bar{\tau}}=L_{\tau, \bar{\tau}}-\frac{1}{2}h^{\omega}(L_{\tau, \bar{\tau}},L_{\tau, \bar{\tau}})Z^{\omega}.$
\end{enumerate}
\end{lemma}
\begin{proof}
    First, using identity (\ref{100125A}), we compute
    $
    \tau\left(K_{\tau, \bar{\tau}}\right)=\mathbf{h}(\eta^{\tau}-\eta^{\overline{\tau}}, \eta^{\tau})=-\mathbf{h}(\eta^{\tau},\eta^{\overline{\tau}}),
    $
    which establishes the first equality in $(1)$. Applying the same identity with $\bar{\tau}$ in place of  $\tau$ yields the second equality. To prove $(2)$, note that since $\Phi$ is an isometry, we have 
    $$
    h^{\omega}\big(W,K_{\tau, \bar{\tau}}\big)=\mathbf{h}(\Phi(W),\eta^{\tau}-\eta^{\overline{\tau}} )= \tau(W)- \bar{\tau}(W),
    $$
    as claimed. Finally, to prove $(3)$,  we deduce from $(2)$ that
 $$
 P^{\bar{\tau}}(K_{\tau, \bar{\tau}})=\sum_{i=1}^{m}h^{\omega}(K_{\tau, \bar{\tau}}, E_{i})E_{i}=\sum_{i=1}^{m}(\tau(E_{i})- \bar{\tau}(E_{i}))E_{i}=L_{\tau, \bar{\tau}}.
 $$ 
Moreover, as a consequence of $(1)$ and $(2)$, we obtain 
\begin{equation}\label{05072025A}
h^{\omega}(L_{\tau, \bar{\tau}}, L_{\tau, \bar{\tau}})=h^{\omega}( K_{\tau, \bar{\tau}},  K_{\tau, \bar{\tau}})=-2\bar{\tau}( K_{\tau, \bar{\tau}}).
\end{equation} 
Therefore, we arrive at the identity in $(3)$.
\end{proof}

\begin{proposition}\label{17102024D}
Let $(\mathcal{T},\mathbf{h},\nabla^{\mathcal{T}}, \xi)$ be the standard tractor bundle constructed from a lightlike Cartan geometry $(p: \mathcal{P}\to N,\omega)$, and let $T \in \Gamma(\mathcal{T})$ be a smooth section. Given two 1-forms $\tau, \overline{\tau} \in S(N, h^{\omega}, Z^{\omega})$, consider the expressions of $T$ with respect to the $\tau$ and $\overline{\tau}$ decompositions of $\mathcal{T}$ described in ${\rm (\ref{01072025A})}$, denoted by $T^{\tau}$ and $T^{\overline{\tau}}$, respectively. Then, the coordinate representations of $T$ in these two splittings are related by
$$T^{\overline{\tau}}=F_{\tau,\overline{\tau}}\cdot T^{\tau},$$ where
\begin{equation}\label{27102024R}
    F_{\tau,\overline{\tau}}:=\begin{pmatrix}
		1 & \overline{\tau}(\cdot) & -\frac{1}{2}h^{\omega}(L_{\tau, \bar{\tau}},L_{\tau, \bar{\tau}})\\
		0 & P^{\overline{\tau}}(\cdot) & L_{\tau, \bar{\tau}}\\
		0 & 0 & 1
	\end{pmatrix}
\end{equation}
and $P^{\overline{\tau}}$ is the map given in ${\rm (\ref{13112024})}$ for $\overline{\tau}$.

\end{proposition}

\begin{proof}
We aim to construct the natural vector bundle isomorphism $F_{\tau,\overline{\tau}}$ that relates the two decompositions of $\mathcal{T}$. That is, we have to relate $T\overset{\tau}{= }\alpha\xi+\Phi(X)+\beta\eta^{\tau}$ with $T\overset{\overline{\tau}}{= }\overline{\alpha}\xi+\Phi(\bar{X})+\overline{\beta}\eta^{\overline{\tau}}.$  The proof follows from a direct computation.

\begin{enumerate}
    \item Considering both  decompositions of $T$, we have $\beta\overset{\tau}{= }\mathbf{h}(T,\xi)\overset{\overline{\tau}}{= }\bar{\beta}$.
    \item  Analogously, we obtain $\alpha+\mathbf{h}(\Phi(X),\eta^{\overline{\tau}})+\beta\mathbf{h}(\eta^{\tau},\eta^{\overline{\tau}})\overset{\tau}{= }\mathbf{h}(T,\eta^{\overline{\tau}})\overset{\overline{\tau}}{= }\overline{\alpha}$. From (\ref{100125A}), we know that $\mathbf{h}(\Phi(X),\eta^{\overline{\tau}})=\overline{\tau}(X)$, and Lemma \ref{27102024D} gives $\mathbf{h}(\eta^{\tau},\eta^{\overline{\tau}})=\overline{\tau}\left(K_{\tau, \bar{\tau}}\right).$ Finally, Equation $(\ref{05072025A})$ yields $\overline{\tau}\left(K_{\tau, \bar{\tau}}\right)=-\frac{1}{2}h^{\omega}(L_{\tau, \bar{\tau}},L_{\tau, \bar{\tau}}).$
    \item Using again both decompositions of $T$, we find
    \begin{align*}
  \Phi(\bar{X}) & = (\alpha-\overline{\alpha})\xi+\Phi(X)+\beta(\eta^{\tau}-\eta^{\overline{\tau}})=(-\overline{\tau}(X)-\beta\mathbf{h}(\eta^{\tau},\eta^{\overline{\tau}}))\xi+\Phi(X)+\beta(\eta^{\tau}-\eta^{\overline{\tau}}) \\
  & =\Phi(P^{\overline{\tau}}(X))+\beta\Phi(L_{\tau, \bar{\tau}}).
\end{align*}
  \end{enumerate}
\end{proof}

\begin{remark}\label{10072025t}
{\rm The map $F_{\tau,\overline{\tau}}$ is an isometry between the two decompositions. That is, for all $T_1,T_2\in\Gamma(\mathcal{T})$, $$\mathbf{h}\big(F_{\tau,\overline{\tau}}\cdot T_1^{\tau},F_{\tau,\overline{\tau}}\cdot T_2^{\tau}\big)=\mathbf{h}(T_1^{\tau}, T_2^{\tau}).$$}
\end{remark}

We now turn our attention to the expression of the tractor connection under these splittings of the standard tractor bundle, with the goal of identifying the geometric structures encoded within.

\begin{proposition}\label{27102024A}
Let $(\mathcal{T},\mathbf{h},\nabla^{\mathcal{T}}, \xi)$ be the standard tractor bundle constructed from a lightlike Cartan geometry $(p: \mathcal{P}\to N,\omega)$. Given a $1$-form $\tau\in S(N,h^{\omega},Z^{\omega})$, the tractor connection $\nabla^{\mathcal{T}}$ takes the following form in the decomposition associated with $\tau$:
\begin{equation}\label{17102024A}
\nabla^{\mathcal{T}}_W\begin{pmatrix}\alpha \\ X \\  \beta\end{pmatrix}\overset{\tau}{=}\begin{pmatrix}W\big(\alpha\big)+\alpha\,\tau\big(W\big)-h^{\omega}\big(X,D^{\tau}(W)\big) \\ \alpha\,P^{\tau}(W)+\beta\,D^{\tau}(W)+\nabla^{\tau}_W X \\  W\big(\beta\big)-\beta\,\tau\big(W\big)-h^{\omega}\big(X,W\big)\end{pmatrix},
\end{equation}
where $\nabla^{\tau}:\mathfrak{X}(N)\times \Gamma(\mathrm{An}\, (\tau)) \to  \Gamma(\mathrm{An}\, (\tau))$ is a metric linear connection on the vector bundle $\left(\mathrm{An}\, (\tau)\rightarrow N,h^{\omega}\right)$, and $D^{\tau}: TN \to \mathrm{An}\, (\tau)$ is a vector bundle morphism.
\end{proposition}
\begin{proof}
From $(\ref{02072025A})$, it follows that
$$
\nabla^{\mathcal{T}}_W\begin{pmatrix}1 \\ 0 \\  0\end{pmatrix}=\begin{pmatrix}\tau(W) \\ P^{\tau}(W) \\  0\end{pmatrix}
$$
for all $W\in\mathfrak{X}(N)$. Since we can write $\nabla^{\mathcal{T}}_W \eta^{\tau}\overset{\tau}{=}\alpha\xi+\Phi(X)+\beta\eta^{\tau}$, with $\alpha,\beta\in\mathcal{C}^{\infty}(N,\R)$ and $X\in\Gamma\left(\mathrm{An}\, (\tau)\right)$, we obtain $\beta\overset{\tau}{=}\mathbf{h}\left(\nabla^{\mathcal{T}}_W \eta^{\tau},\xi\right)=-\mathbf{h}\left(\Phi(W),\eta^{\tau}\right)$. By Equation $(\ref{100125A})$, we conclude that $\beta=-\tau(W)$. Similarly, we compute $\alpha=0$. Hence, we deduce
$$
\nabla^{\mathcal{T}}_W\begin{pmatrix}0 \\ 0 \\  1\end{pmatrix}=\begin{pmatrix}0 \\ D^{\tau}(W)\\  -\tau\big(W\big)\end{pmatrix},
$$
where $D^{\tau}: TN \to \mathrm{An}\, (\tau)$ is necessarily a vector bundle morphism. 
The $\mathcal{C}^{\infty}(N,\R)$-linearity of $D^{\tau}$ follows directly from the linearity of $\nabla^{\mathcal{T}}$. Proceeding similarly, we obtain
$$
\nabla^{\mathcal{T}}_W\begin{pmatrix}0 \\ X \\  0\end{pmatrix}=\begin{pmatrix}-h^{\omega}\big(X,D^{\tau}(W)\big) \\ \nabla^{\tau}_W X \\ -h^{\omega}\big(X,W\big) \end{pmatrix},
$$
for $X\in\Gamma\left(\mathrm{An}\, (\tau)\right)$, where $\nabla^{\tau}:\mathfrak{X}(N)\times \Gamma(\mathrm{An}\, (\tau)) \to  \Gamma(\mathrm{An}\, (\tau))$ is a metric linear connection on the vector bundle $\left(\mathrm{An}\, (\tau)\rightarrow N,h^{\omega}\right)$. The properties of $\nabla^{\tau}$ follow directly from those of $\nabla^{\mathcal{T}}$.
Finally, the above formulas determine the values of $\nabla^{\mathcal{T}}$ on each component of the decomposition of $\mathcal{T}$, and the linearity properties of $\nabla^{\mathcal{T}}$ imply the claimed formula (\ref{17102024A}). 
\end{proof}

\begin{remark}\label{17072025u}
    {\rm As a consequence of Proposition \ref{27102024A}, we obtain 
\begin{equation}\label{03072025A}
R^{\mathcal{T}}(V,W)\xi\overset{\tau}{=}\begin{pmatrix}d\tau (V,W)+\theta^{\tau}(V,W) \\  \nabla^{\tau}_{V}(P^{\tau}(W))-\nabla^{\tau}_{W}(P^{\tau}(V))-P^{\tau}([V,W])-B^{\tau}(V,W)\\  0\end{pmatrix},
\end{equation}
for every $V,W\in\mathfrak{X}(N)$. The $2$-form $\theta^{\tau}\in \Omega^{2}(N ,\R)$ is defined by $\theta^{\tau}(V,W):=h^{\omega}\big(V,D^{\tau}(W))-h^{\omega}\big(W,D^{\tau}(V)\big)$, and the bilinear map $B^{\tau}$ is given by $B^{\tau}(V,W):=\tau(V)W-\tau(W)V$, for all $V,W\in \mathfrak{X}(N)$.}
\end{remark}

\begin{remark}\label{21072025j}
{\rm Recall that a Galilean spacetime $(M, \Omega, g, \nabla)$ is a Leibnizian spacetime $(M, \Omega, g)$ endowed with a Galilean connection $\nabla$ on $M$, see \cite{SB03}. That is, $\nabla$ is a linear connection satisfying $\nabla\Omega=0$ and
$$
V\big(g(X,Y)\big)= g(\nabla_{V}X,Y)+g(X,\nabla_{V}Y),\quad V\in \mathfrak{X}(M), \, X,Y\in\Gamma(\mathrm{An}\,(\Omega)).
$$
Galilean connections are not uniquely determined by the structure $(M, \Omega, g)$. A key result in \cite[Theorem 5.27]{SB03} provides an explicit characterization of the space of such connections. Given a Leibnizian spacetime $(M, \Omega, g)$, let $\mathcal{D}(\Omega, g)$ denote the set of Galilean connections compatible with the structure. Then, for any vector field $Z\in \mathfrak{X}(M)$ such that $\Omega(Z)=1$, the map
$$
\mathcal{D}(\Omega,g)\to\Gamma(\mathrm{An}\,(\Omega))\times\Lambda^{2}(\mathrm{An}\,(\Omega))\times\Lambda^{2}(TM,\mathrm{An}\,(\Omega)),\quad \nabla\mapsto\left( \nabla_{Z}Z,\ \dfrac{1}{2}\mathrm{rot\, Z},\ P^{\Omega}\circ\mathrm{Tor}\right)
$$
is a bijection. Here, $\mathrm{Tor}(V,W):=\nabla_{V}W - \nabla_{W}V-[V,W]$ is the torsion tensor of $\nabla$, and the vorticity along $Z$ is the skew-symmetric tensor defined by $(\mathrm{rot\, Z})(X,Y):=g(\nabla_{X}Z,Y)-g(X,\nabla_{Y}Z)$ for $ X,Y\in \Gamma (\mathrm{An}\,(\Omega))$. The vector field $\nabla_Z Z$ is called the gravitational field of $\nabla$. Furthermore, as shown in \cite{SB03}, every Galilean connection $\nabla$ satisfies the identity $$\Omega\circ \mathrm{Tor}=d\Omega. $$

Now, in our setting, for each $\tau\in S(N,h^{\omega},Z^{\omega})$, we can extend the connection $\nabla^{\tau}$ to a Galilean connection $\widetilde{\nabla}^{\tau}$ on $(N,\tau,h^{\omega})$, given by 
\begin{equation}\label{150825C}
    \widetilde{\nabla}^{\tau}_V W:=V\big(\tau(W)\big)Z^{\omega}+\tau(W)P^{\tau}(V)+\nabla^{\tau}_V \big(P^{\tau}(W)\big).
\end{equation}
Let us note that the decomposition (\ref{17102024A}) can be written as
$$
\nabla^{\mathcal{T}}_{V}\Phi(W)\overset{\tau}{=}\Phi\Big(\widetilde{\nabla}^{\tau}_{V}W+\tau(V)\tau(W)Z^{\omega}-h^{\omega}(D^{\tau}(V),W)Z^{\omega}\Big)- h^{\omega}(V,W)\eta^{\tau}.
$$
The connection $\widetilde{\nabla}^{\tau}$ has vanishing gravitational field with respect to $Z^{\omega}$ and zero vorticity along $Z^{\omega}$. On the other hand, its torsion satisfies 
$$
  \big(P^{\tau}\circ\mathrm{\widetilde{Tor}^{\tau}}\big)(V,W)=\nabla^{\tau}_{V}(P^{\tau}(W))-\nabla^{\tau}_{W}(P^{\tau}(V))-P^{\tau}([V,W])-B^{\tau}(V,W),
$$
where $B^{\tau}$ is defined in Remark \ref{17072025u}. Consequently, expression $(\ref{03072025A})$ becomes
\begin{equation}\label{150825A}
R^{\mathcal{T}}(V,W)\xi\overset{\tau}{=}\begin{pmatrix}\big(\tau\circ\mathrm{\widetilde{Tor}^{\tau}}\big)(V,W)+\theta^{\tau}(V,W) \\ \big(P^{\tau}\circ\mathrm{\widetilde{Tor}^{\tau}}\big)(V,W) \\  0\end{pmatrix}.
\end{equation}
Note that it is enough to determine $P^{\tau}\circ\mathrm{\widetilde{Tor}^{\tau}}$ to uniquely determine the connection $\widetilde{\nabla}^{\tau}$, and therefore also uniquely determines the connection $\nabla^{\tau}$ since
$$\nabla^{\tau}_{V}X=\widetilde{\nabla}^{\tau}_{V}X,\quad X\in \Gamma(\mathrm{An}\, (\tau)),\,\, V\in \mathfrak{X}(N).$$ }
\end{remark}

From Proposition \ref{27102024A}, each $\tau\in S(N,h^{\omega},Z^{\omega})$ determines, via the tractor connection $\nabla^{\mathcal{T}}$, a linear connection $\nabla^{\tau}$ on the subbundle $\mathrm{An}\,(\tau)\rightarrow N$ that is compatible with the restriction of the metric $h^{\omega}$ to $\mathrm{An}\,(\tau)$, as well as a vector bundle morphism $D^{\tau}: TN \to \mathrm{An}\, (\tau)$. The following result provides explicit formulas relating the connections $\nabla^{\tau},\nabla^{\overline{\tau}}$ and the morphisms $D^{\tau},D^{\overline{\tau}}$ associated with two different choices $\tau,\overline{\tau}\in S(N,h^{\omega},Z^{\omega})$.

\begin{lemma}\label{27102024U}
Let $(\mathcal{T}, \mathbf{h},\nabla^{\mathcal{T}}, \xi)$ be the standard tractor bundle  constructed from a lightlike Cartan geometry $(p: \mathcal{P}\to N,\omega)$, and let $\tau,\overline{\tau}\in S(N,h^{\omega},Z^{\omega})$. Then, for any $W\in\mathfrak{X}(N)$ and $X\in\Gamma\left(\mathrm{An}\, (\tau)\right),$ we have:
\begin{enumerate}

\item[$(1)$] $\nabla^{\overline{\tau}}_W (P^{\overline{\tau}}\left(X\right)) = P^{\overline{\tau}}\Big(\nabla^{\tau}_W X-\overline{\tau}(X)W\Big)-h^{\omega}(X,W)L_{\tau, \bar{\tau}}.$

\item[$(2)$] $D^{\overline{\tau}}(W)  = P^{\overline{\tau}}\Big(D^{\tau}(W)+\dfrac{1}{2}h^{\omega}(L_{\tau, \bar{\tau}},L_{\tau, \bar{\tau}})W\Big)-\tau(W)L_{\tau, \bar{\tau}}- \nabla^{\overline{\tau}}_W L_{\tau, \bar{\tau}}.$
\end{enumerate}
\end{lemma}

\begin{proof} Let us denote by $\nabla^{\mathcal{T}_{\tau}}$ and $\nabla^{\mathcal{T}_{\overline{\tau}}}$ the decompositions of the tractor connection $\nabla^{\mathcal{T}}$ with respect to $\tau$ and $\overline{\tau}$, respectively, as in (\ref{17102024A}).
Then, for every section $T\overset{\tau}{= }\alpha\xi+\Phi(X)+\beta\eta^{\tau}$, the following identity holds: 
\begin{equation}\label{070325B}
F_{\tau,\overline{\tau}}\cdot\nabla^{\mathcal{T}_{\tau}}_W\begin{pmatrix}\alpha \\ X \\  \beta\end{pmatrix}=\nabla^{\mathcal{T}_{\overline{\tau}}}_W \left(F_{\tau,\overline{\tau}}\cdot\begin{pmatrix}\alpha \\ X \\  \beta\end{pmatrix}\right),
\end{equation}
where $F_{\tau,\overline{\tau}}$ is the transformation defined in $(\ref{27102024R})$. For $T\overset{\tau}{= }\xi$, identity (\ref{070325B}) yields no non-trivial information.
 For $T\overset{\tau}{= }\Phi(X)$, Identity (\ref{070325B}) is equivalent to 
$$
  \nabla^{\overline{\tau}}_W (P^{\overline{\tau}}\left(X\right)) = P^{\overline{\tau}}\Big(\nabla^{\tau}_W X-\overline{\tau}(X)W\Big)-h^{\omega}(X,W)L_{\tau, \bar{\tau}}
$$
and 
\begin{align}\label{07072025A}
  h^{\omega}(X,D^{\overline{\tau}}(W)) & = h^{\omega}(X,D^{\tau}(W))+\overline{\tau}(X)\overline{\tau}(W)-\frac{1}{2}h^{\omega}(L_{\tau, \bar{\tau}},L_{\tau, \bar{\tau}})h^{\omega}(X,W)\notag \\
  & -h^{\omega}\big(X,\nabla^{\tau}_W \big(P^{\tau}(L_{\tau,\overline{\tau}})\big)\big).
\end{align}
Moreover, for $T\overset{\tau}{= }\eta^{\tau}$,  Identity (\ref{070325B}) is equivalent to  
$$
D^{\overline{\tau}}(W)  = P^{\overline{\tau}}\Big(D^{\tau}(W)+\frac{1}{2}h^{\omega}(L_{\tau, \bar{\tau}},L_{\tau, \bar{\tau}})W\Big)-\tau(W)L_{\tau, \bar{\tau}}- \nabla^{\overline{\tau}}_W L_{\tau, \bar{\tau}}
$$
and
\begin{equation}\label{07072025B}
\tau\big(D^{\overline{\tau}}(W)\big)+\overline{\tau}\big(D^{\tau}(W)\big)=-\frac{1}{2}W\big(h^{\omega}(L_{\tau, \bar{\tau}},L_{\tau, \bar{\tau}})\big)-\frac{1}{2}\big(\tau(W)+\overline{\tau}(W)\big)h^{\omega}(L_{\tau, \bar{\tau}},L_{\tau, \bar{\tau}}).
\end{equation}
Equations (\ref{07072025A}) and (\ref{07072025B}) can be derived by direct computation from $(1)$, $(2)$, and the metric compatibility of $\nabla^{\tau}$ with $(\mathrm{An}\,(\tau)\rightarrow N, h^{\omega})$.
\end{proof}

The next result reveals the underlying geometric structure of a Cartan geometry modeled on the future lightlike cone in Lorentz-Minkowski spacetime.
\begin{theorem}\label{10072025A}
    Let $(p: \mathcal{P}\to N,\omega)$ be a lightlike Cartan geometry. Then, the manifold $N$ is endowed with the following structures:
    \begin{enumerate}
        \item[$(1)$] A lightlike metric $h^{\omega}$ and a vector field $Z^{\omega}$ spanning its radical distribution. 
        \item[$(2)$] A map $\nabla^{\omega}$ that assigns to each $\tau \in S(N,h^{\omega},Z^{\omega})$ a metric linear connection $\nabla^{\tau}$ on the vector bundle $(\mathrm{An}\, (\tau)\to N,h^{\omega})$, such that for any $\tau, \bar{\tau}\in S(N,h^{\omega},Z^{\omega})$ and any section $X\in\Gamma\left( \mathrm{An}\, (\tau)\right)$, the following identity holds:
    \begin{equation}\label{16072025J}
  \nabla^{\overline{\tau}}_W (P^{\overline{\tau}}\left(X\right)) = P^{\overline{\tau}}\Big(\nabla^{\tau}_W X-\overline{\tau}(X)W\Big)-h^{\omega}(X,W)L_{\tau, \bar{\tau}}.
   \end{equation}
         \item[$(3)$] A map $D^{\omega}$ that assigns to each $\tau \in S(N,h^{\omega},Z^{\omega})$ a vector bundle morphism $D^{\tau}: TN \to \mathrm{An}\, (\tau)$, such that for any $\tau, \bar{\tau}\in S(N,h^{\omega},Z^{\omega})$, the following identity holds:
    \begin{equation}\label{16072025P}
        D^{\overline{\tau}}(W)  = P^{\overline{\tau}}\Big(D^{\tau}(W)+\frac{1}{2}h^{\omega}(L_{\tau, \bar{\tau}},L_{\tau, \bar{\tau}})W\Big)-\tau(W)L_{\tau, \bar{\tau}}- \nabla^{\overline{\tau}}_W L_{\tau, \bar{\tau}}.
    \end{equation}
    \end{enumerate}   
\end{theorem}

\begin{remark}\label{21072025oki}
{\rm Let us note that part of the map $\nabla^{\omega}$ is entirely determined by the intrinsic structure of the lightlike manifold $(N,h^{\omega},Z^{\omega})$. This follows from Identity $(\ref{16072025D})$, since
$$
2h^{\omega}\big(A_{Z^{\omega}}[V],[W]\big)=\big(\mathcal{L}_{Z^{\omega}} h^{\omega}\big)([V],[W])=Z^{\omega}\big(h^{\omega}(V,W)\big)-h^{\omega}([Z^{\omega},V],W)-h^{\omega}([Z^{\omega},W],V),
$$
for all $V,W\in\mathfrak{X}(N)$. Therefore, for each $\tau\in S(N,h^{\omega},Z^{\omega})$, we have
\begin{equation}\label{21072025mjh}
2h^{\omega}\big(A_{Z^{\omega}}[V],[W]\big)=h^{\omega}\big(\nabla^{\tau}_{Z^{\omega}}(P^{\tau}(V))-[Z^{\omega},V],W\big)+h^{\omega}\big(\nabla^{\tau}_{Z^{\omega}}(P^{\tau}(W))-[Z^{\omega},W],V\big).
\end{equation}
Now, define the endomorphism $J([V]):=\big[\nabla^{\tau}_{Z^{\omega}}\big(P^{\tau}(V)\big)-[Z^{\omega},V]\big]$. By Identity $(\ref{16072025J})$, the map $J$ does not depend on the choice of $\tau$. Next, consider the symmetric part of $J$, defined by
$$
h^{\omega}(J_{\mathrm{sym}}([V]),[W])=\dfrac{1}{2}\big(h^{\omega}(J([V]),[W])+h^{\omega}([V],J([W]))\big)=h^{\omega}\big(A_{Z^{\omega}}[V],[W]\big).
$$
Therefore, we conclude that $J_{\mathrm{sym}} = A_{Z^{\omega}}$.}
\end{remark}

\begin{remark}\label{130825A}
    {\rm A particularly useful feature for constructing examples is that the maps $\nabla^{\omega}$ and $D^{\omega}$ are completely determined by their value on a single $1$-form, as a consequence of Equations $(\ref{16072025J})$ and $(\ref{16072025P})$. Fix $\alpha \in S(N,h^{\omega},Z^{\omega})$ and assign to it a metric linear connection $\nabla^{\alpha}$ on $(\mathrm{An}\, (\alpha)\to N,h^{\omega})$ together with a vector bundle morphism $D^{\alpha}: TN \to \mathrm{An}\, (\alpha)$. For each $\tau \in S(N,h^{\omega},Z^{\omega})$, define the metric linear connection $\nabla^{\tau}$ on $(\mathrm{An}\, (\tau)\to N,h^{\omega})$ by 
$$
      \nabla^{\tau}_W (P^{\tau}\left(X\right)) = P^{\tau}\Big(\nabla^{\alpha}_W X-\tau(X)W\Big)-h^{\omega}(X,W)L_{\alpha, \tau},
$$
where $X\in\Gamma\left( \mathrm{An}\, (\alpha)\right)$. It is important to note that $\left.P^{\tau}\right|_{\Gamma\left(\mathrm{An}\, (\alpha)\right)}$ is an isomorphism onto $\Gamma\left(\mathrm{An}\, (\tau)\right)$. Similarly, define the vector bundle morphism $D^{\tau}: TN \to \mathrm{An}\, (\tau)$ by
$$
         D^{\tau}(W)  = P^{\tau}\Big(D^{\alpha}(W)+\frac{1}{2}h^{\omega}(L_{\alpha, \tau},L_{\alpha, \tau})W\Big)-\alpha(W)L_{\alpha, \tau}- \nabla^{\tau}_W L_{\alpha, \tau}.
$$
Then, for every \(\tau,\overline{\tau} \in S(N,h^{\omega},Z^{\omega})\), Equations $(\ref{16072025J})$ and $(\ref{16072025P})$ hold. In fact, for every $X\in\Gamma\left( \mathrm{An}\, (\tau)\right)$ we have
\begin{align*}  
  \nabla^{\overline{\tau}}_W\big(P^{\overline{\tau}}(X)\big)  &= P^{\overline{\tau}}\big(\nabla^{\alpha}_W\big(P^{\alpha}(X)\big)-\overline{\tau}(P^{\alpha}(X))W\big)-h^{\omega}(X,W)L_{\alpha,\overline{\tau}}  \\
    &= P^{\overline{\tau}}\big(\nabla^{\tau}_WX-\overline{\tau}(X)W\big)+h^{\omega}(X,W)\big(P^{\overline{\tau}}\big(L_{\alpha,\tau}\big)-L_{\alpha,\overline{\tau}}\big).
\end{align*}
Therefore, Equation $(\ref{16072025J})$ holds if and only if
$$
    P^{\overline{\tau}}(L_{\alpha, \tau}) - L_{\alpha, \overline{\tau}} = -L_{\tau, \overline{\tau}}
$$
and this identity follows directly from Equation (\ref{090825D}). To verify that  $(\ref{16072025P})$ holds, it suffices to perform the following calculation:
\begin{align*}  
    D^{\overline{\tau}}(W)&=P^{\overline{\tau}}\Big(D^{\alpha}(W)+\frac{1}{2}h^{\omega}(L_{\alpha, \overline{\tau}},L_{\alpha, \overline{\tau}})W\Big)-\alpha(W)L_{\alpha, \overline{\tau}}- \nabla^{\overline{\tau}}_W L_{\alpha, \overline{\tau}} \\
    &=P^{\overline{\tau}}\Big(D^{\tau}(W)-\frac{1}{2}h^{\omega}(L_{\alpha, \tau},L_{\alpha, \tau})W+\alpha(W)L_{\alpha,\tau}+\nabla^{\tau}_W L_{\alpha,\tau}+\frac{1}{2}h^{\omega}(L_{\alpha, \overline{\tau}},L_{\alpha, \overline{\tau}})W\Big)\\
    & -\alpha(W)L_{\alpha,\overline{\tau}}-\nabla^{\overline{\tau}}_W L_{\alpha,\overline{\tau}}.
\end{align*}
From Equation (\ref{090825D}) we have
\begin{align*}  
    \nabla^{\overline{\tau}}_W L_{\alpha,\overline{\tau}}&= \nabla^{\overline{\tau}}_W\big(P^{\overline{\tau}}(L_{\alpha,\tau})\big)+\nabla^{\overline{\tau}}_W L_{\tau,\overline{\tau}}\\
    &=P^{\overline{\tau}}\big(\nabla^{\tau}_W L_{\alpha,\tau}-\overline{\tau}(L_{\alpha,\tau})W\big)-h^{\omega}(L_{\alpha,\tau},W)L_{\tau,\overline{\tau}}+\nabla^{\overline{\tau}}_W L_{\tau,\overline{\tau}}.
\end{align*}
Combining the last two equations, we obtain
\begin{align*}  
    D^{\overline{\tau}}(W)&=P^{\overline{\tau}}\big(D^{\tau}(W)\big)-\nabla^{\overline{\tau}}_W L_{\tau,\overline{\tau}}+h^{\omega}(L_{\alpha,\tau},W)L_{\tau,\overline{\tau}}-\alpha(W)L_{\alpha,\overline{\tau}} \\
    &+P^{\overline{\tau}}\Big(-\frac{1}{2}h^{\omega}(L_{\alpha, \tau},L_{\alpha, \tau})W+\alpha(W)L_{\alpha,\tau}+\frac{1}{2}h^{\omega}(L_{\alpha, \overline{\tau}},L_{\alpha, \overline{\tau}})W+\overline{\tau}(L_{\alpha,\tau})W\Big).
\end{align*}
Therefore, Equation $(\ref{16072025P})$ follows directly from (\ref{090825D}) and Lemma \ref{27102024D}. }
\end{remark}

Lemma \ref{27102024U} can equivalently be reformulated in terms of the Galilean connections $\widetilde{\nabla}^{\tau}$ introduced in Equation $(\ref{150825C})$, as follows.
\begin{lemma}\label{080825A}
Let $(\mathcal{T}, \mathbf{h},\nabla^{\mathcal{T}}, \xi)$ be the standard tractor bundle  constructed from a lightlike Cartan geometry $(p: \mathcal{P}\to N,\omega)$, and let $\tau,\overline{\tau}\in S(N,h^{\omega},Z^{\omega})$. Then, for any $V, W\in\mathfrak{X}(N)$ we have:
\begin{enumerate}

\item[$(1)$] $\widetilde{\nabla}^{\overline{\tau}}_V W = P^{\overline{\tau}}\Big(\widetilde{\nabla}^{\tau}_V W\Big)+V(\bar{\tau}(W))Z^{\omega}-h^{\omega}(V,W)L_{\tau, \bar{\tau}}.$

\item[$(2)$] $D^{\overline{\tau}}(W)  = P^{\overline{\tau}}\Big(D^{\tau}(W)+\dfrac{1}{2}h^{\omega}(L_{\tau, \bar{\tau}},L_{\tau, \bar{\tau}})W\Big)-\tau(W)L_{\tau, \bar{\tau}}- \widetilde{\nabla}^{\overline{\tau}}_W L_{\tau, \bar{\tau}}.$
\end{enumerate}
\end{lemma}
\begin{proof}
The proof is a straightforward computation from Lemma \ref{27102024U}. 
\end{proof}

\begin{remark}
{\rm For every $\tau \in S(N, h^{\omega}, Z^{\omega})$, Equation $(\ref{150825A})$ yields
\begin{equation}\label{070825A}
    \mathbf{T}^{\omega}(V,W)=(P^{\tau}\circ\mathrm{\widetilde{Tor}^{\tau}})(V,W)+(d\tau(V,W)+ \theta^{\tau}(V,W))Z^{\omega},\quad \text{for all }V,W\in\mathfrak{X}(N),
\end{equation}
where $\mathbf{T}^{\omega}\in \Gamma(\Lambda^{2}T^{*}N\otimes TN)$ is the tensor introduced in (\ref{040825A}). In particular, the right-hand side of this expression is independent of the choice of $\tau$. Consequently, the class $\big[\mathbf{T}^{\omega}(V,W)\big]=\big[(P^{\tau}\circ\mathrm{\widetilde{Tor}^{\tau}})(V,W)\big]\in \Gamma(\mathcal{E})$ does not depend on the Galilean connection $\widetilde{\nabla}^{\tau}$. Moreover, combining Remark \ref{21072025j} with Equation $(\ref{070825A})$, we see that $\mathbf{T}^{\omega}$ uniquely determines all Galilean connections $\widetilde{\nabla}^{\tau}$ through the conditions
\begin{equation}\label{050825A}
\widetilde{\nabla}^{\tau}_{Z^{\omega}}Z^{\omega}=0, \quad \mathrm{rot\, }Z^{\omega}=0,\quad \mathrm{\widetilde{Tor}^{\tau}}=P^{\tau}\circ \mathbf{T}^{\omega}+ d\tau \otimes Z^{\omega}.
\end{equation}}
\end{remark}

\begin{remark}
{\rm Note that, if we define
$$
    \widetilde{\nabla}^{\overline{\tau}}_V W := P^{\overline{\tau}}\Big(\widetilde{\nabla}^{\tau}_V W\Big)+V(\bar{\tau}(W))Z^{\omega}-h^{\omega}(V,W)L_{\tau, \bar{\tau}},
$$
where $\widetilde{\nabla}^{\tau}$ is the Galilean connection determined by Conditions (\ref{050825A}), then $\widetilde{\nabla}^{\overline{\tau}}$ also satisfies the corresponding Conditions (\ref{050825A}). For instance, one readily checks that
$$
   \mathrm{\widetilde{Tor}^{\overline{\tau}}}(V,W)=P^{\overline{\tau}}\Big(\mathrm{\widetilde{Tor}^{\tau}}(V,W)\Big)+ d\overline{\tau} (V,W)Z^{\omega}=P^{\overline{\tau}}\Big(\mathbf{T}^{\omega}(V,W)\Big)+ d\overline{\tau} (V,W)Z^{\omega}.
$$}
\end{remark}
The next result reveals the underlying geometric structure of a Cartan geometry modeled on the future lightlike cone in Lorentz-Minkowski spacetime in terms of Galilean connections. It is equivalent to Theorem \ref{10072025A}.
\begin{theorem}\label{19082025hfjrfhgfujfg}
    Let $(p: \mathcal{P}\to N,\omega)$ be a lightlike Cartan geometry. Then, the manifold $N$ is endowed with the following structures:
    \begin{enumerate}
        \item[$(1)$] A lightlike metric $h^{\omega}$ and a vector field $Z^{\omega}$ spanning its radical distribution. 
         \item[$(2)$] A tensor field $\mathbf{T}^{\omega}\in\Gamma(\Lambda^{2}T^{*}N \otimes TN)$, which, for each $\tau \in S(N,h^{\omega},Z^{\omega})$, determines the Galilean connection $\widetilde{\nabla}^{\tau}$ through Conditions ${\rm (\ref{050825A})}$.
        \item[$(3)$] A map $D^{\omega}$ that assigns to each $\tau \in S(N,h^{\omega},Z^{\omega})$ a vector bundle morphism $D^{\tau}: TN \to \mathrm{An}\, (\tau)$, such that for any $\tau, \bar{\tau}\in S(N,h^{\omega},Z^{\omega})$, the following identity holds:
        \begin{equation}\label{23082025ytu}
        D^{\overline{\tau}}(W)  = P^{\overline{\tau}}\Big(D^{\tau}(W)+\dfrac{1}{2}h^{\omega}(L_{\tau, \bar{\tau}},L_{\tau, \bar{\tau}})W\Big)-\tau(W)L_{\tau, \bar{\tau}}- \widetilde{\nabla}^{\overline{\tau}}_W L_{\tau, \bar{\tau}}.
        \end{equation}
    \end{enumerate}   
\end{theorem}

\subsection{Construction of a lightlike extension vector bundle over $N$}\label{14082025utight}

In this Subsection, we construct a lightlike extension vector bundle over $N$ from the structures introduced in Theorem \ref{10072025A}. That is, we aim to show that these structures are sufficient to determine a unique lightlike Cartan geometry. The proof follows a standard procedure (see \cite{Baum}), so we will include only a summary for the sake of completeness. To this end, we introduce the following definition.
\begin{definition}\label{06122024A}
    Let $(N,h,Z)$ be a lightlike manifold.  We define a lightlike-compatible structure as a pair $(\nabla,D)$ satisfying the following conditions:
    \begin{itemize}
 
\item[$(1)$] The map $\nabla$ assigns to each $\tau \in S(N,h,Z)$ a metric linear connection $\nabla^{\tau}$ on the vector bundle $(\mathrm{An}\, (\tau)\to N,h)$, such that Identity $(\ref{16072025J})$ is satisfied.
        
\item[$(2)$] The map $D$ assigns to each $\tau \in S(N,h,Z)$ a vector bundle morphism $D^{\tau}: TN \to \mathrm{An}\, (\tau)$, such that Identity $(\ref{16072025P})$ is satisfied.
\end{itemize}
\end{definition}
From now on, let $(N,h,Z)$ be a lightlike manifold equipped with a lightlike-compatible structure $(\nabla,D)$. For each $\tau\in S(N,h,Z)$,  we construct a vector bundle associated with the data $(N,h,Z,\nabla^{\tau},D^{\tau})$ endowed with a linear connection and a bundle-like metric:
\begin{enumerate}
	\item Let $\mathcal{T}_{\tau}$ be the rank $m+2$ real vector bundle over $N$ given by
	$$
	\mathcal{T}_{\tau}:=\underline{\R}\oplus \mathrm{An}\,(\tau)\oplus \underline{\R},
	$$
	where $\underline{\R}$ denotes the trivial bundle $N\times\R\rightarrow N$. For any section $T^{\tau}\in\Gamma\left(\mathcal{T}_{\tau}\right)$ we write
	$$
	T^{\tau}=\begin{pmatrix}\alpha \\ X \\  \beta\end{pmatrix},\,\,\text{with } \alpha,\beta\in\mathcal{C}^{\infty}(N,\R) \,\,\text{and}\,\,X\in\Gamma(\mathrm{An}\,(\tau)).
	$$
	\item On $\mathcal{T}_{\tau}$ we consider the linear connection $\nabla^{\mathcal{T}_{\tau}}$ defined by
	$$
	\nabla^{\mathcal{T}_{\tau}}_W\begin{pmatrix}\alpha \\ X \\  \beta\end{pmatrix}:=\begin{pmatrix}W\big(\alpha\big)+\alpha\,\tau\big(W\big)-h(X,D^{\tau}(W)) \\ \alpha\,P^{\tau}(W)+\beta\,D^{\tau}(W)+\nabla^{\tau}_W X \\  W\big(\beta\big)-\beta\,\tau\big(W\big)-h(X,W)\end{pmatrix},
	$$
	where $W\in\mathfrak{X}(N)$, and the bundle-like metric $\mathbf{h}^{\tau}$ of Lorentzian signature given by
	$$
	\mathbf{h}^{\tau}\left(\begin{pmatrix}\alpha_1 \\ X_1 \\  \beta_1\end{pmatrix},\begin{pmatrix}\alpha_2 \\ X_2 \\  \beta_2\end{pmatrix}\right):=\alpha_1\beta_2+\beta_1\alpha_2+h(X_1,X_2).
	$$
It is straightforward to verify that the linear connection $\nabla^{\mathcal{T}_{\tau}}$ is metric for $\mathbf{h}^{\tau}$. The triple $(\mathcal{T}_{\tau},\mathbf{h}^{\tau},\nabla^{\mathcal{T}_{\tau}})$ is referred to as the $\tau$-tractor.
\end{enumerate}
 
Now, given $\tau,\overline{\tau}\in S(N,h,Z)$, we consider the corresponding $\tau$-tractor and $\overline{\tau}$-tractor. In view of Proposition \ref{17102024D}, Remark \ref{10072025t} and Lemma \ref{27102024U}, the following properties hold:
\begin{itemize}
\item[$(a)$] Every section $T^{\overline{\tau}}\in\Gamma\left(\mathcal{T}_{\overline{\tau}}\right)$ can be written as $T^{\overline{\tau}}=F_{\tau,\overline{\tau}}\cdot T^{\tau}$ for a uniquely determined section $T^{\tau}\in\Gamma\left(\mathcal{T}_{\tau}\right)$. In other words, the map $F_{\tau,\overline{\tau}}:\mathcal{T}_{\tau}\rightarrow\mathcal{T}_{\overline{\tau}}$ defined in $(\ref{27102024R})$ is a vector bundle isomorphism.

\item[$(b)$] $\mathbf{h}^{\overline{\tau}}(F_{\tau,\overline{\tau}}\cdot T_1^{\tau},F_{\tau,\overline{\tau}}\cdot T_2^{\tau})=\mathbf{h}^{\tau}(T^{\tau}_1,T^{\tau}_2),$ where $T^{\tau}_1,T^{\tau}_2\in\Gamma\left(\mathcal{T}_{\tau}\right)$.

\item[$(c)$] $F_{\tau,\overline{\tau}}\cdot\nabla^{\mathcal{T}_{\tau}}_W T^{\tau}=\nabla^{\mathcal{T}_{\overline{\tau}}}_W \left(F_{\tau,\overline{\tau}}\cdot T^{\tau}\right),$ where $T^{\tau}\in\Gamma\left(\mathcal{T}_{\tau}\right)$ and $W\in\mathfrak{X}(N)$.
\end{itemize}
As a consequence of $(a)$, one can define $\hat{\mathcal{T}}:=\displaystyle\underset{\scriptscriptstyle\tau\in S(N,h,Z)}{\bigcup}\mathcal{T}_{\tau}$ and introduce the following equivalence relation:
$$
\begin{pmatrix}a \\ X \\  b\end{pmatrix}\in\left(\mathcal{T}_{\tau}\right)_x\sim \begin{pmatrix}\overline{a} \\\overline{X} \\  \overline{b}\end{pmatrix}\in\left(\mathcal{T}_{\overline{\tau}}\right)_x\Longleftrightarrow\begin{pmatrix}\overline{a} \\\overline{X} \\  \overline{b}\end{pmatrix}=F_{\tau,\overline{\tau}}(x)\cdot\begin{pmatrix}a \\ X \\  b\end{pmatrix},
$$
for $x\in N$. We denote the resulting quotient set by $\mathcal{T}:=\hat{\mathcal{T}}/\hspace{-0.25em}\sim$. By interpreting each natural projection $\pi_{\tau}\colon \mathcal{T}_{\tau}\to \mathcal{T}$ as a vector bundle isomorphism, we endow $\mathcal{T}\rightarrow N$ with a well-defined canonical vector bundle structure over $N$. Note that $F_{\tau,\overline{\tau}}=\pi_{\overline{\tau}}^{-1}\circ \pi_{\tau}$. Moreover, in view of $(b)$, the vector bundle $\mathcal{T}\to N$ admits a natural bundle-like metric $\mathbf{h}$ of Lorentzian signature, uniquely determined by the requirement that each projection $\pi_{\tau}\colon \mathcal{T}_{\tau}\to \mathcal{T}$ be an isometry. This definition is consistent, since the maps $F_{\tau,\overline{\tau} }$ are themselves isometries. Note that the vector bundle  $\mathcal{T}\rightarrow N$ also admits a natural linear connection $\nabla^{\mathcal{T}}$ defined by
$$
\nabla^{\mathcal{T}}_W T:=\pi_{\tau}\Big(\nabla^{\mathcal{T}_{\tau}}_W \big(\pi_{\tau}^{-1}(T)\big)\Big),\quad W\in\mathfrak{X}(N).
$$
By $(c)$, the definition of $\nabla^{\mathcal{T}}$
is independent of the choice of $\tau \in S(N,h,Z)$. It is clear that  $\nabla^{\mathcal{T}}$ is metric with respect to the bundle-like metric $\mathbf{h}$. Additionally, we obtain a well-defined distinguished lightlike section of $\mathcal{T}$ given by
    $$
\xi:=\pi_{\tau}\begin{pmatrix}1 \\ 0 \\ 0\end{pmatrix}\in\Gamma(\mathcal{T}).
    $$ 
Since
    $$
    \nabla^{\mathcal{T}}_{W}\xi=\pi_{\tau}\begin{pmatrix}\tau(W) \\ P^{\tau}(W)\\ 0\end{pmatrix},
    $$
    the assignment  $
    \Phi(W):=\nabla^{\mathcal{T}}_{W}\xi
    $
    defines a vector bundle monomorphism from $TN$ into $\mathcal{T}$. Therefore, we obtain the claimed lightlike extension vector bundle over $N$.

Hence, we have constructed a lightlike extension vector bundle  from a lightlike manifold $(N,h,Z)$ equipped with a lightlike-compatible structure $(\nabla,D)$. It is straightforward to verify that this construction is the inverse procedure to that described in Subsection \ref{27102024p}. This means that the structures introduced in Theorem \ref{10072025A} are in bijective correspondence with lightlike Cartan geometries. Thus, all the above can be summarized in the following result.

\begin{theorem}\label{140825A}
There is a bijective correspondence between lightlike Cartan geometries and lightlike manifolds $(N,h,Z)$ equipped with a lightlike-compatible structure $(\nabla,D)$.

\bigskip
$$
			\begin{tikzpicture}
				
				\node (A) at (0,0) {$(p:\mathcal{P}\to N, \omega)$};
				\node (B) at (4,0) {$(\mathcal{T},\mathbf{h},\nabla^{\mathcal{T}},\xi)$};
				\node (C) at (2,-1.5) {$(N,h,Z,\nabla,D)$};
				
				\draw[<->] (A) -- (B);
				\draw[<->] (B) -- (C);
				\draw[<->] (C) -- (A);
			\end{tikzpicture}
$$

\end{theorem}

\subsection{Examples}\label{14082025mqsi}

Let $(M^{2n+1}, g)$ be an odd-dimensional Riemannian manifold with $n\geq 1$ that admits a Killing vector field $Z \in\mathfrak{X}(M)$. Define $\varphi:=-\nabla^{g} Z$, where $\nabla^g$ denotes the Levi-Civita connection of $g$. The triple $(M,g,Z)$ is called a Sasakian manifold if:
\begin{enumerate}
    \item $g(Z,Z)=1.$
    \item $\varphi^{2}=-\mathrm{Id}+g(Z,\cdot)Z.$
    \item $\big(\nabla^{g}_{V}\varphi\big)(W)=g(V,W)Z-g(W,Z)V,$\quad  $V,W\in \mathfrak{X}(M)$.
\end{enumerate}
These conditions imply several additional relations:
$$
d\eta(V,W) = 2g(V, \varphi(W)),  
\quad \big(\nabla_{V}^{g}\eta\big)(W) = g(V, \varphi(W)), \quad g(V,W) = g(\varphi(V), \varphi(W)) + \eta(V)\eta(W),
$$
where $\eta:=g(Z,\cdot)$ is the $1$-form metrically equivalent to the vector field $Z$. The Sasakian form is the $2$-form defined by $\Psi(V,W):=g(V,\varphi(W))$; for a general reference on Sasakian manifolds, see \cite{Blair}.

The generalized Tanaka connection was introduced in \cite{Tanno}. It is a $g$-metric connection in the broader class of contact metric manifolds, given by
$$
\nabla^{*}_{V}W=\nabla^{g}_{V}W+\eta(V)\varphi(W)-\eta(W)\nabla^{g}_{V}Z+\big(\nabla^{g}_{V}\eta\big)(W)Z,
$$
with the additional condition $\nabla^{*}\eta=0$ (equivalently, $\nabla^{*}Z=0)$. For an excellent reference on metric connections with torsion and their applications, see \cite{Agr}. In the Sasakian case, the generalized Tanaka connection reduces to
$$
\nabla^{*}_{V}W=\nabla^{g}_{V}W+\eta(V)\varphi(W)+\eta(W)\varphi(V)+\Psi(V,W) Z.
$$
Every Sasakian manifold naturally carries the lightlike metric
$$
h(V,W):=g(\varphi(V), \varphi(W))=g(V,W)-\eta(V)\eta(W)
$$
whose radical is globally spanned by the vector field $Z$. 

We now equip $(M,h,Z)$ with a lightlike-compatible structure $(\nabla,D)$. First, for $\eta\in S(M,h,Z)$, we define $\nabla^{\eta}$ to be the restriction of $\nabla^{*}$ to $\mathrm{An}\,(\eta)$, so that
$$
\nabla^{\eta}_{V}X:=\nabla^{g}_{V}X+\eta(V)\varphi(X)+\Psi(V,X)Z, \quad X\in \Gamma\left(\mathrm{An}\, (\eta)\right).
$$
Since $\nabla^{*}$ is $g$-metric and $\nabla^{*}\eta=0$, a straightforward computation shows that $\nabla^{\eta}$ is a metric linear connection on $(\mathrm{An}\, (\eta)\to M,h)$. According to Remark \ref{130825A}, for every $\tau \in S(M,h,Z)$ and $X\in \Gamma\left(\mathrm{An}\, (\eta)\right)$, the map $\nabla$ is given by
$$
\tau\in S(M,h,Z)\mapsto \nabla^{\tau}, \quad \nabla^{\tau}_{V}\big(P^{\tau}(X)\big)= P^{\tau}\Big(\nabla^{g}_V X+\eta(V)\varphi(X)-\tau(X)V\Big)-h(X,V)L_{\eta, \tau}.
$$
Similarly, we set $D^{\eta}:=\varphi$, so that
$$
         D^{\tau}(V)  = P^{\tau}\Big(\varphi(V)+\frac{1}{2}h(L_{\eta, \tau},L_{\eta, \tau})V\Big)-\eta(V)L_{\eta, \tau}- \nabla^{\tau}_V L_{\eta, \tau}.
$$
Theorem \ref{140825A} then provides a lightlike Cartan geometry associated with the Sasakian manifold $(M,g,Z)$.

We can determine the associated vector bundle endomorphism $A_{Z}\colon \mathcal{E}\to \mathcal{E}$ from the following sequence of computations:
$$
(\mathcal{L}_{Z}h)(V,W)=(\mathcal{L}_{Z}g)(V,W)-(\mathcal{L}_{Z}(\eta \otimes \eta))(V,W)=-(\mathcal{L}_{Z}(\eta \otimes \eta))(V,W).
$$
A direct calculation yields
$$
(\mathcal{L}_{Z}\eta)(V)=(di_{Z}\eta)(V)+(i_{Z}d\eta)(V)=d\eta(Z,V)=2g(Z, \varphi(V))=0.
$$
Therefore, the vector bundle endomorphism $A_Z$ is identically zero (see Equation $(\ref{16072025D})$). This aligns perfectly with Remark \ref{21072025oki}, as we have $J = 0$ and, in particular, $J_{\mathrm{sym}} = 0$.

As mentioned in Remark \ref{21072025j}, every connection $\nabla^{\tau}$ induces a Galilean connection defined by
$$
\widetilde{\nabla}^\tau _V W=V\big(\tau(W)\big)Z+\tau(W)P^{\tau}(V)+\nabla^{\tau}_V \big(P^{\tau}(W)\big).
$$
In particular, for the Galilean connection associated with $\nabla^{\eta}$ we obtain
\begin{align*}
\widetilde{\nabla}^\eta _V W &=V\big(\eta(W)\big)Z+\eta(W)P^{\eta}(V)+\nabla^{\eta}_V \big(P^{\eta}(W)\big)\\
& =\nabla^{*}_{V}W +\eta(W)P^{\eta}(V).
\end{align*}
It should be noted that this connection does not coincide with the generalized Tanaka connection.

In a similar way, one can construct a lightlike Cartan geometry on a degenerate hyperplane of Lorentz–Minkowski spacetime as follows. Let $\Pi\subset\L^{m+2}$ be a degenerate hyperplane. In suitable coordinates $(r_{0}, \dots,r_{m})$ on $\Pi$, the induced metric takes the form
$$
h=dr_{1}^{2}+\dots+dr_{m}^{2},
$$
with the lightlike vector field $Z = \frac{\partial}{\partial r_{0}}$. Consider $\alpha:=dr_{0}\in S(\Pi, h, Z)$. According to Remark \ref{130825A}, it suffices to define
$$
\nabla^{\alpha}_{V}X:=\sum_{i=1}^{m} V(x_{i}) \frac{\partial}{\partial r_{i}}, 
$$
where $X=\sum\limits_{i=1}^m x_{i}\frac{\partial}{\partial r_{i}}\in\Gamma\left(\mathrm{An}\,(\alpha)\right)$,
and
$$
D^{\alpha}(V):=P^{\alpha}(V).
$$
It is straightforward to check that $\nabla^{\alpha}$ is a metric linear connection on $(\mathrm{An}\,(\alpha)\rightarrow \Pi,h)$, and that $\mathcal{L}_{Z} h = 0$.

\section{A normalization for lightlike manifolds with $A_{Z}=\mathrm{Id}$}\label{14082025ythgirthndokgv}

In this Section, we present a normalization procedure for a subclass of lightlike Cartan geometries. This construction yields a simplified description of the geometry under suitable assumptions. As in the general framework, the normalization conditions for lightlike Cartan geometries single out a unique lightlike Cartan geometry (up to isomorphism) associated with a lightlike manifold $(N,h,Z)$.
However, it does not encompass the entire class: certain geometries fall outside the scope of this normalization, see Remark $\ref{21072025ghj}$. Despite this limitation, the method captures a substantial and geometrically meaningful subclass, which we analyze in detail below. The approach is inspired by the normalization procedure for normal Cartan connections in conformal geometry. Due to structural similarities between the conformal and lightlike models, an adaptation of this procedure proves effective for the subclass considered here. 
For the formulation of the normalization conditions in the conformal setting within the framework of tractor bundles, see \cite{CG03}; see also \cite{MP1} for a related treatment. For a broader and more general perspective, see \cite{CS09}.

Let $(N,h,Z)$ be a lightlike manifold. In Section \ref{27102024y} we proved that giving a lightlike extension vector bundle $(\mathcal{T},\mathbf{h},\nabla^{\mathcal{T}},\xi)$ for $(N,h,Z)$ is equivalent to giving a lightlike Cartan geometry $(p: \mathcal{P}\to N, \omega)$ such that $h^{\omega}=h$ and $Z^{\omega}=Z$, in the sense that the former arises as the standard tractor bundle associated with the latter. In Section \ref{14082025847564857} we showed that the tractor connection $\nabla^{\mathcal{T}}$ uniquely determines a lightlike–compatible structure $(\nabla,D)$, obtained via the decomposition described in Proposition \ref{27102024A}. 
\begin{proposition}\label{21072025y}
Let $(N,h,Z)$ be a lightlike manifold, and let $(\mathcal{T},\mathbf{h},\nabla^{\mathcal{T}},\xi)$ be a lightlike extension vector bundle for $(N,h,Z)$. Then, the following statements are equivalent:
\begin{itemize}
\item[$(1)$] $R^{\mathcal{T}}(V,W)\xi$ is collinear with $\xi$ for every $V,W\in\mathfrak{X}(N)$.
\item[$(2)$] For each $\tau\in S(N,h,Z)$, the map $\nabla$ satisfies:
\begin{enumerate}
\item[$(2.1)$] $\nabla^{\tau}_Z X=X+P^{\tau}\big([Z,X]\big)$, for every $X\in\Gamma\left( \mathrm{An}\, (\tau)\right)$.
\item[$(2.2)$] $\nabla^{\tau}_X Y-\nabla^{\tau}_Y X-P^{\tau}\big([X,Y]\big)=0$, for every $X,Y\in\Gamma\left( \mathrm{An}\, (\tau)\right)$.
\end{enumerate}
\item[$(3)$] For each $\tau\in S(N,h,Z)$, the corresponding Galilean connection $\widetilde{\nabla}^{\tau}$ defined in $(\ref{150825C})$ satisfies $P^{\tau}\circ\mathrm{\widetilde{Tor}^{\tau}}=0.$
\end{itemize}
In particular, the condition that $R^{\mathcal{T}}(V,W)\xi$ is collinear with $\xi$ uniquely determines the map $\nabla$.
\end{proposition}

\begin{proof}
$(1)\Rightarrow(2):$ From Equation $(\ref{03072025A})$, we obtain 
\begin{equation}\label{21072025V}
\nabla^{\tau}_{V}(P^{\tau}(W))-\nabla^{\tau}_{W}(P^{\tau}(V))-P^{\tau}([V,W])=B^{\tau}(V,W)
\end{equation}
for every $V,W\in\mathfrak{X}(N)$. Now, taking $V=Z$ and $W\in\Gamma\left( \mathrm{An}\, (\tau)\right)$, we recover $(2.1)$, while choosing $V,W\in\Gamma\left( \mathrm{An}\, (\tau)\right)$ yields $(2.2)$. $(2)\Rightarrow(1):$ It is a straightforward computation to verify that Equation  $(\ref{21072025V})$ follows from $(2.1)$ and $(2.2)$. The equivalence $(1)\Leftrightarrow(3)$ follows directly from Equation (\ref{150825A}). The last statement is a consequence of \cite[Theorem 5.27]{SB03}; see  Remark $\ref{21072025j}$.
\end{proof}

\begin{remark}\label{21072025ghj}
{\rm It is important to emphasize that the normalization presented in Proposition $\ref{21072025y}$ does not hold in general. As noted in Remark $\ref{21072025oki}$, the map $\nabla$ contains a component determined by the intrinsic structure of $(N,h,Z)$. In fact, if $R^{\mathcal{T}}(V,W)\xi$ is collinear with $\xi$, then, by Proposition $\ref{21072025y}$ and Equation $(\ref{21072025mjh})$, it follows that $A_Z$ must be the identity map. That is, $Z$ must necessarily be a homothetic vector field. Therefore, such a normalization condition can only be imposed in this specific setting. Note that if $Z$ is a conformal vector field satisfying $\mathcal{L}_Z h=2fh$ for some non-vanishing function $f\in\mathcal{C}^{\infty}(N,\R)$, one can always rescale $Z$ by defining $\hat{Z}:=\frac{1}{f}Z$, so that $\hat{Z}$ becomes homothetic.}
\end{remark}

\begin{remark}\label{nabla1234}
{\rm The condition that $R^{\mathcal{T}}(V,W)\xi$ is collinear with $\xi$ is independent of the choice of any $\tau\in S(N,h,Z)$. Consequently, the map $\nabla$ defined in Proposition $\ref{21072025y}(2)$ must satisfy Equation $(\ref{16072025J})$. Nevertheless, a direct verification is sufficiently interesting to merit a brief sketch of the proof. In \cite[Lemma 5.25]{SB03}, a Koszul-like formula for Galilean connections is provided. Applying this formula to the Galilean connection $\widetilde{\nabla}^{\tau}$ described in Remark \ref{21072025j}, and assuming that $R^{\mathcal{T}}(V,W)\xi$ is collinear with $\xi$, we obtain the classical Koszul formula but adapted to our family of connections $\nabla^{\tau}$. Specifically, for each $\tau\in S(N,h,Z)$, we have:
\begin{align*}
 2h(\nabla^{\tau}_W X,Y) & =W\big(h(X,Y)\big)+X\big(h(W,Y)\big)-Y\big(h(X,W)\big) \\
  &+h([W,X],Y)-h([X,Y],W)+h([Y,W],X),
\end{align*}
for every $X,Y\in\Gamma\left(\mathrm{An}\, (\tau)\right)$ and $W\in\mathfrak{X}(N).$ Now, given $\tau,\overline{\tau}\in S(N,h,Z)$, it suffices to use this Koszul formula to compute $2h\big(\nabla^{\overline{\tau}}_W \big(P^{\overline{\tau}}(X)\big),P^{\overline{\tau}}(Y)\big)$ for $X,Y\in\Gamma\left(\mathrm{An}\, (\tau)\right)$, and thereby deduce the change Equation $(\ref{16072025J})$.}
\end{remark}

\begin{remark}\label{14082025hgjdkld}
{\rm Properly totally umbilical lightlike hypersurfaces in Lorentzian manifolds constitute a broad class of examples where the condition $R^{\mathcal{T}}(V,W)\,\xi$ being collinear with $\xi$ may consistently occur. Consider an immersion $\psi:(N,h,Z)\rightarrow (M,g)$, where $(M,g)$ is a Lorentzian manifold and the lightlike metric is given by $h=\psi^*g$. The Levi-Civita connection $\nabla^g$ of $(M,g)$ gives rise to the induced connection $\overline{\nabla}^g:\mathfrak{X}(N)\times\overline{\mathfrak{X}}(N)\rightarrow\overline{\mathfrak{X}}(N)$, where $\overline{\mathfrak{X}}(N)$ denotes the space of vector fields along the immersion $\psi$; see \cite[Chap. 4]{One83}. Moreover, in this context one can introduce the null-Weingarten map with respect to $Z$, defined as $$[\overline{\nabla}^g Z]:\mathcal{E}\rightarrow \mathcal{E}, \quad [V]\mapsto [\overline{\nabla}^g_V Z],$$
as well as the null second fundamental form, given by $$B_Z([V],[W]):=h\big([\overline{\nabla}^g Z][V],[W]\big).$$
As shown in \cite[Section 5]{PacoLuz}, in this setting we have $A_Z=[\overline{\nabla}^g Z]$. Furthermore, a lightlike hypersurface is said to be properly totally umbilical if there exists a non-vanishing function $f$ on $N$ such that $B_Z=fh$. Consequently, the properly totally umbilical condition is equivalent to the vector field $Z$ being conformal. Taking into account Remark \ref{21072025ghj}, we conclude that the condition $(1)$ in Proposition \ref{21072025y} is admissible for every properly totally umbilical lightlike hypersurface.}
\end{remark}
We now focus on the conditions that determine the map 
$D$.
\begin{proposition}\label{26067872}
Let $(N,h,Z)$ be a lightlike manifold, and let $(\mathcal{T},\mathbf{h},\nabla^{\mathcal{T}},\xi)$ be a lightlike extension vector bundle for $(N,h,Z)$. Assume that $R^{\mathcal{T}}(V,W)\xi$ is collinear with $\xi$ for every $V,W\in\mathfrak{X}(N)$.
 Then, the following statements are equivalent:
    \begin{itemize}
     \item[$(1)$] $R^{\mathcal{T}}(Z,V)\Phi(W)$ is collinear with $\xi$ for every $V,W\in\mathfrak{X}(N)$.
     \item[$(2)$] For each $\tau \in S(N,h,Z)$, the curvature term $R^{\mathcal{T}}(Z,X)\Phi(Y)$ is collinear with $\xi$ for every $X,Y\in \Gamma\left(\mathrm{An}\, (\tau)\right)$.
     \item[$(3)$] For each $\tau\in S(N,h,Z)$, the map $D$ satisfies the following equation
     \begin{equation}\label{300720258}
h\big(X,Y\big)D^{\tau}(Z)-h\big(Y,D^{\tau}(Z)\big)X=R^{\tau}(Z,X)Y,
     \end{equation}
     where $R^{\tau}(V,W)X:=\nabla^{\tau}_V\nabla^{\tau}_W X-\nabla^{\tau}_W\nabla^{\tau}_V X-\nabla^{\tau}_{[V,W]} X$ for all $V,W\in\mathfrak{X}(N)$ and $X,Y\in\Gamma\left( \mathrm{An}\, (\tau)\right).$
    \end{itemize}
\end{proposition}
\begin{proof}
Clearly, $(1)\Leftrightarrow (2)$; it suffices to note that
\begin{align*}
R^{\mathcal{T}}(Z,V)\Phi(W)&=R^{\mathcal{T}}\Big(Z,P^{\tau}(V)+\tau(V)Z\Big)\Phi(P^{\tau}(W)+\tau(W)Z)\\
&=R^{\mathcal{T}}\Big(Z,P^{\tau}(V)\Big)\Phi(P^{\tau}(W))+\tau(W)R^{\mathcal{T}}\Big(Z,P^{\tau}(V)\Big)\xi.
\end{align*}
Now, using $(\ref{17102024A})$, we obtain
\begin{align*}
  R^{\mathcal{T}}(Z,X)\Phi(Y) & \overset{\tau}{=} \nabla^{\mathcal{T}}_Z \begin{pmatrix} -h\big(Y,D^{\tau}(X)\big)\\ \nabla^{\tau}_X Y \\ -h\big(X,Y\big)  \end{pmatrix} -\nabla^{\mathcal{T}}_X \begin{pmatrix} -h\big(Y,D^{\tau}(Z)\big)\\\nabla^{\tau}_Z Y  \\  0 \end{pmatrix} -\begin{pmatrix} -h\big(Y,D^{\tau}([Z,X])\big) \\ \nabla^{\tau}_{[Z,X]}\big(Y\big) \\ -h\big(Y,P^{\tau}([Z,X])\big)  \end{pmatrix}.
\end{align*}
Thus, again using $(\ref{17102024A})$, we find that the third component  is given by
$$
-Z\big(h\big(X,Y\big)\big)+h\big(X,Y\big)+h\big(\nabla^{\tau}_Z Y,X\big)+h\big(Y,P^{\tau}([Z,X])\big).
$$
Since $R^{\mathcal{T}}(V,W)\xi$ is collinear with $\xi$, as a consequence of Proposition \ref{21072025y}$(2.1)$, we conclude that the third coordinate vanishes identically. Therefore, the condition stated in point $(1)$ is equivalent to the vanishing of the second component of $R^{\mathcal{T}}(Z,X)\Phi(Y)$. A direct computation shows that this is equivalent to the following equation:
$$
h\big(X,Y\big)D^{\tau}(Z)-h\big(Y,D^{\tau}(Z)\big)X=R^{\tau}(Z,X)Y.
$$
\end{proof}

\begin{remark}\label{22082025ryt}
    {\rm Equation (\ref{300720258}) has at most one solution. Indeed, letting $(E_{1}, \dots, E_{m})$ be a local frame of $\mathrm{An}\, (\tau)$ such that $h(E_{i},E_{j})=\delta_{ij}$, we deduce that Equation (\ref{300720258}) implies 
$$
(m-1)D^{\tau}(Z)=\displaystyle\sum_{i=1}^{m}R^{\tau}(Z,E_i)E_i.
$$
As far as we can verify, in the absence of additional assumptions on the curvature, the existence of a solution cannot be guaranteed. Therefore, the existence of a lightlike extension vector bundle for which $R^{\mathcal{T}}(Z,V)\Phi(W)$ is collinear with $\xi$ cannot, in general, be ensured. 
According to Remark \ref{190825A}, for lightlike manifolds that are locally bundles of scales of Riemannian conformal manifolds, condition $(1)$ in Proposition \ref{26067872} is satisfied.
}
\end{remark}

From now on, assume that we are given a lightlike extension vector bundle $(\mathcal{T},\mathbf{h},\nabla^{\mathcal{T}},\xi)$ for $(N,h,Z)$ such that $R^{\mathcal{T}}(V,W)\xi$ and $R^{\mathcal{T}}(Z,V)\Phi(W)$ are collinear with $\xi$ for every $V,W\in\mathfrak{X}(N)$. Under this assumptions, we can define the map $\mathcal{W}\in\Gamma\left(\Lambda^2\mathcal{E}^*\otimes L(\mathcal{E},\mathcal{E})\right)$ by
$$
\mathcal{W}([U],[V])[W]:=\left[\Phi^{-1}\big(R^{\mathcal{T}}(U,V)\Phi(W)\big)\right].
$$
The map $\mathcal{W}$ is well-defined as a consequence of the collinearity conditions stated above, and the following Ricci-type contraction is well-posed:
$$\mathrm{trace}\big([U]\mapsto \mathcal{W}([U],[V])[W]\big).
$$
Now, given a $1$-form $\tau\in S(N,h,Z)$ and  a local frame $(E_{1}, \dots, E_{m})$  of $\mathrm{An}\,(\tau)$ such that $h(E_{i},E_{j})=\delta_{ij}$, this Ricci-type contraction can be computed as
\begin{equation}\label{26072025t}
\displaystyle\sum_{i=1}^m h\big(\mathcal{W}([E_i],[V])[W],[E_i]\big)=\displaystyle\sum_{i=1}^m h\left(\left[\Phi^{-1}\big(R^{\mathcal{T}}(E_i,V)\Phi(W)\big)\right],[E_i]\right).
\end{equation}
Hence, computing the Ricci-type contraction in $(\ref{26072025t})$ is equivalent to computing
$$
\displaystyle\sum_{i=1}^m h\left(\Phi^{-1}\big(R^{\mathcal{T}}(E_i,X)\Phi(Y)\big),E_i\right),
$$
where $X,Y\in\Gamma\left(\mathrm{An}\, (\tau)\right)$. 
As a consequence of Equation $(\ref{17102024A})$, we obtain the following expression for the Ricci-type contraction:
\begin{align*}
  \displaystyle\sum_{i=1}^m h\left(\Phi^{-1}\big(R^{\mathcal{T}}(E_i,X)\Phi(Y)\big),E_i\right) & = -(m-2)h(D^{\tau}(X),Y)+ \mathrm{Ric}^{\tau}(X,Y)-h(X,Y)\mathrm{tr_{\tau} }D^{\tau},
\end{align*}
where 
$$
\mathrm{tr_{\tau} }D^{\tau}:=\sum\limits_{i=1}^m h(D^{\tau}(E_i),E_i)\quad  \textrm{ and } \quad \mathrm{Ric}^{\tau}(X,Y):=\sum\limits_{i=1}^m h\big(R^{\tau}(E_{i},X) Y,E_i\big).
$$
In general, the tensor $\mathrm{Ric}^{\tau}$ is not necessarily symmetric, in contrast with the Riemannian case.

\begin{proposition}\label{263473621a}
Let $(N,h,Z)$ be a lightlike manifold, and let $(\mathcal{T},\mathbf{h},\nabla^{\mathcal{T}},\xi)$ be a lightlike extension vector bundle for $(N,h,Z)$. Assume that $m\geq 3$, and that $R^{\mathcal{T}}(V,W)\xi$ and $R^{\mathcal{T}}(Z,V)\Phi(W)$ are collinear with $\xi$ for every $V,W\in\mathfrak{X}(N)$. Then, the following statements are equivalent:
  \begin{itemize}
     \item[$(1)$] The Ricci-type contraction $(\ref{26072025t})$ vanishes.
     \item[$(2)$] For each $\tau\in S(N,h,Z)$, the map $D$ satisfies
     \begin{equation}\label{12234ewrfwe}
     h(D^{\tau}(X),Y)=\dfrac{1}{m-2}\left(\mathrm{Ric}^{\tau}(X,Y)-\dfrac{S^{\tau}}{2(m-1)}h(X,Y)\right),
     \end{equation}
    for every $X,Y\in\Gamma\left(\mathrm{An}\, (\tau)\right)$, where $S^{\tau}:= \sum\limits_{i=1}^m \mathrm{Ric}^{\tau}(E_i,E_i)$ and $(E_{1}, \dots, E_{m})$ is a local frame of $\mathrm{An}\, (\tau)$ such that $h(E_{i},E_{j})=\delta_{ij}$.
    \end{itemize}
\end{proposition}
\begin{proof}
From the computations above, it is clear that $(1)$ is equivalent to
\begin{equation}\label{26662266}
(m-2)h(D^{\tau}(X),Y)= \mathrm{Ric}^{\tau}(X,Y)-h(X,Y)\mathrm{tr_{\tau} }D^{\tau}.
\end{equation}
Taking the trace of both sides with respect to the frame $(E_{1}, \dots, E_{m})$, we obtain
$$
\mathrm{tr_{\tau} }D^{\tau}=\dfrac{S^{\tau}}{2(m-1)}.
$$
Substituting this expression into Equation $(\ref{26662266})$, we get
$$
(m-2)h(D^{\tau}(X),Y)= \mathrm{Ric}^{\tau}(X,Y)-\dfrac{S^{\tau}}{2(m-1)}h(X,Y),
$$
which completes the proof.
\end{proof}
\begin{remark}\label{dhfihrihfihrihs}
{\rm It is worth noting that Equation $(\ref{12234ewrfwe})$ closely resembles the expression of the Schouten tensor associated with a Riemannian metric. This motivates our definition $
S^{\tau}:= \sum\limits_{i=1}^m \mathrm{Ric}^{\tau}(E_i,E_i),$ so that $S^{\tau}$ plays a role analogous to that of the scalar curvature in the classical definition of the Schouten tensor for a Riemannian metric.}
\end{remark}
\begin{theorem}\label{21082025trg}
Let \((N^{m+1}, h, Z)\) be a lightlike manifold with \(m \geq 3\), and assume that \(A_Z = \mathrm{Id}\). Then there exists at most one lightlike extension vector bundle \((\mathcal{T}, \mathbf{h}, \nabla^{\mathcal{T}}, \xi)\) for $(N^{m+1},h,Z)$ satisfying the following conditions:
\begin{enumerate}
    \item[$(1)$] \(R^{\mathcal{T}}(V, W)\xi\) is collinear with \(\xi\) for every \(V, W \in \mathfrak{X}(N)\).
    \item[$(2)$] $R^{\mathcal{T}}(Z,V)\Phi(W)$ is collinear with $\xi$ for every $V,W\in\mathfrak{X}(N)$.
    \item[$(3)$] The Ricci-type contraction of \(\mathcal{W} \in \Gamma\left(\Lambda^2 \mathcal{E}^* \otimes L(\mathcal{E}, \mathcal{E})\right)\) vanishes.
\end{enumerate}
Moreover, such a lightlike extension vector bundle exists if and only if Equation $(\ref{300720258})$ is solvable for every $\tau\in S(N,h,Z)$. 
\end{theorem}

\begin{remark}
    {\rm According to Section \ref{27102024y}, lightlike Cartan geometries can be equivalently described in terms of tractor bundles. Therefore, the result above establishes the uniqueness of a lightlike Cartan geometry $(p\colon \mathcal{P} \to N, \omega)$ such that $h^{\omega} = h$ and $Z^{\omega} = Z$, if it exists. This Cartan geometry corresponds naturally to the lightlike extension vector bundle $(\mathcal{T}, \mathbf{h}, \nabla^{\mathcal{T}}, \xi)$, which satisfies the technical assumptions of Theorem \ref{21082025trg}.  }
\end{remark}

\smallskip

\begin{remark}\label{010820256A}
{\rm Let $(\mathcal{L},h,Z)$ be the fiber bundle of scales of the Riemannian conformal manifold $(M,[g])$. Thus, the projection $\pi:\mathcal{L}\rightarrow M$ defines a principal fiber bundle with structure group $\R_{>0}$. There is a natural bijection between global sections of the scale bundle and Riemannian metrics in the conformal class $[g]$, given by
$$
\Gamma\big(\mathcal{L}\big)\,\longleftrightarrow [g],\, \quad s\longleftrightarrow s^*h.
$$
For any two sections $s_1,s_2\in\Gamma\big(\mathcal{L}\big)$, there exists a function $u\in\mathcal{C}^{\infty}(M,\R)$ such that $s_2=e^{2u}s_1$. As a result, the corresponding metrics on $M$ satisfy $s_2^*h=e^{2u}s_1^*h.$ Furthermore, each global section $s\in\Gamma(\mathcal{L})$ induces a global trivialization $\Psi_s:\mathcal{L}  \longrightarrow M\times\R_{>0}$, and a $1$-form $\tau\in S(\mathcal{L},h,Z)$ defined by $\tau:=\Psi_s^{*}dt$, where $t$ denotes the standard coordinate on $\R_{>0}$. It is worth noting that, in this case, $\mathrm{An}\, (\tau)$ defines an integrable distribution.

With this in mind, it becomes clear that the geometric structures introduced throughout this Section correspond to well-established objects in conformal geometry. In the special case where $R^{\mathcal{T}}(V,W)\xi$ is collinear with $\xi$, Equation (\ref{16072025J}), restricted to directions transverse to $Z$, naturally descends to $M$ through sections of the scale bundle. This reproduces the classical relation between the Levi-Civita connections of conformally related metrics \cite{Besse}. Similarly, Equation (\ref{16072025P}) recovers the standard conformal transformation law satisfied by the endomorphisms associated with the Schouten tensors of two conformally related metrics. Furthermore, when $R^{\mathcal{T}}(Z,V)\Phi(W)$ is collinear with $\xi$ and the Ricci-type contraction $(\ref{26072025t})$ vanishes, the tensor defined in  $(\ref{12234ewrfwe})$ corresponds precisely to the Schouten tensor. These observations confirm that the normalization conditions introduced here for lightlike Cartan geometries are fully compatible with those of normal conformal Cartan connections \cite{CG03}.}
\end{remark}

\begin{remark}
{\rm As established in Theorem $\ref{21082025trg}$, for $m\geq 3$, the imposed normalization conditions uniquely determine the map $D$ whenever Equation $(\ref{300720258})$ is solvable. However, in dimension $m=2$, the situation differs, as $D(X)$ is not determined by $(\ref{12234ewrfwe})$. In light of Remark $\ref{010820256A}$ and given the correspondence between the tensor $(\ref{12234ewrfwe})$ and the Schouten tensor, it is natural to expect that Schouten-type tensors (see \cite{M1}) and Möbius structures (see \cite{MP2} and \cite{Ran}) will play a central role in the study of the case $m=2$, just as they do in two-dimensional conformal geometry, where the Schouten tensor is not well-defined.}
\end{remark}

\vspace{2mm}

\noindent {\bf Availability of data and materials:} No data or materials are associated with this study.

\vspace{2mm}

\noindent{\bf Declarations Conflict of interest:} The authors state that there is no conflict of interest.

\end{document}